\documentclass[a4paper,10pt,reqno]{amsart}
\usepackage{amssymb,amsmath,amsthm}
\usepackage{stackrel}
\usepackage{url}
\usepackage{color}
\usepackage[shortlabels]{enumitem}
\usepackage{mathtools}
\mathtoolsset{showonlyrefs}

\newcommand{\ie}{\emph{i.e.}}
\newcommand{\eg}{\emph{e.g.}}
\newcommand{\cf}{\emph{cf.}}
\newcommand{\Real}{\mathbb{R}}
\newcommand{\Com}{\mathbb{C}}
\newcommand{\Nat}{\mathbb{N}}
\newcommand{\Int}{\mathbb{Z}}
\newcommand{\sgn}{\mathop{\mathrm{sgn}}\nolimits}
\newcommand{\supp}{\mathop{\mathrm{supp}}\nolimits}
\newcommand{\Dom}{\mathsf{D}}
\newcommand{\dist}{\mathop{\mathrm{dist}}\nolimits}
\newcommand{\esssup}{\mathop{\mathrm{ess\;\!sup}}}
\newcommand{\eps}{\varepsilon}
\newcommand{\sii}{L^2}

\newcommand{\der}{\mathrm{d}}
\renewcommand{\Re}{\operatorname{Re}}
\renewcommand{\Im}{\operatorname{Im}}
\newcommand{\CcR}{{C_0^{\infty}(\Real)}}

\newcommand{\BigO}{\mathcal{O}}
\newcommand{\ls}{\lesssim}
\newcommand{\gs}{\gtrsim}

\newcommand{\cI}{\mathcal I}
\newcommand{\cJ}{\mathcal J}

\newtheorem{Theorem}{Theorem}[section]
\newtheorem{Lemma}[Theorem]{Lemma}
\newtheorem{Proposition}[Theorem]{Proposition}

\theoremstyle{definition}
\newtheorem{Remark}[Theorem]{Remark}

\newtheorem{Assumption}{Assumption}

\newtheorem{Example}[Theorem]{Example}
\numberwithin{equation}{section}

\author{David Krej\v ci\v r\'ik}
\address[David Krej\v ci\v r\'ik]{Department of Mathematics, Faculty of Nuclear Sciences and Physical Engineering, Czech Technical University in Prague, Trojanova 13, 12000 Prague~2, Czech Republic}
\email{david.krejcirik@fjfi.cvut.cz}
	
\author{Petr Siegl}
\address[Petr Siegl]{
School of Mathematics and Physics, 
Queen's University Belfast, University Road, Belfast BT7 1NN, Northern Ireland, UK}
\email{p.siegl@qub.ac.uk}

\begin{document}
%
\title{Pseudomodes for Schr\"odinger operators with complex potentials}

\date{August 15, 2018}



\begin{abstract}
For one-dimensional Schr\"odinger operators with complex-valued potentials, 
we construct pseudomodes corresponding to large pseudoeigenvalues.
We develop a first systematic non-semi-classical approach,
which results in a substantial progress in achieving 
optimal conditions and conclusions as well as in covering a wide class
of previously inaccessible potentials, including discontinuous ones.
Applications of the present results to higher-dimensional
Schr\"odinger operators are also discussed.
\end{abstract}

\maketitle

%

\section{Introduction}
%
While the spectral theorem reduces the study of self-adjoint operators
to determining the individual components of the spectrum
and the corresponding spectral measures,
it is well known that the spectrum of a non-normal operator
provides by far insufficient information about its properties.
It is not the spectrum that determines the decay of the associated heat semigroup 
and the behaviour of eigenvalues under small perturbations,
but rather the \emph{pseudospectrum}, which measures the largeness of the resolvent, {see \eg~\cite{Trefethen-Embree,Davies_2007,KS-book}.

The $\eps$-pseudospectrum of a closed operator~$H$ consists 
of the union of its spectrum and complex points~$\lambda$ satisfying 
$\|(H-\lambda)f\| < \eps \, \|f\|$ for some vector~$f$ from the domain of~$H$. 
The number~$\lambda$ and the vector~$f$ are respectively called 
the \emph{pseudoeigenvalue} (or \emph{approximate eigenvalue})
and \emph{pseudoeigenvector} (or \emph{pseudomode}) of~$H$.
The pseudoeigenvalues of~$H$ may be turned into genuine eigenvalues 
of a perturbed operator $H+L$ with $\|L\| < \eps$
and they can lie outside (in fact ``very far'' from)
the $\eps$-neighbourhood of 
the spectrum of~$H$ if the operator is not normal.
This is the well-known \emph{spectral instability} of 
non-normal operators under small perturbations.

This paper is concerned with a study, in several aspects complete, 
of approximate eigenvalues and pseudomodes 
of the one-dimensional Schr\"odinger operators
\begin{equation}\label{Hamiltonian.intro}
H_V := -\frac{\der^2}{\der x^2} + V(x)  
\,,
\end{equation}
where $V$ is a \emph{complex}-valued function.
We consider $\sii$-realisations of~$H_V$ 
on the whole line~$\Real$ or the semi-axis~$\Real_+$, the latter having 
immediate consequences for multi-dimensional operators 
with radial potentials and their perturbations.
Thus our objective is to construct 
a $\lambda$-dependent family of pseudomodes~$f_\lambda$ such that 
\begin{equation}\label{objective.intro}
\| (H_V - \lambda) f_\lambda \| = o(1) \, \|f_\lambda\|
\qquad \text{as}\qquad 
\lambda \to \infty \quad \text{in} \quad \Omega \subset \Com.
\end{equation}

The abstract self-adjoint theory yields immediately 
that real-valued potentials~$V$ are irrelevant here, 
since then \eqref{objective.intro} may hold 
only when~$\lambda$ approaches the spectrum of $H_V$.
On the other hand, the by now well-known examples of 
potentials for which \eqref{objective.intro} holds in 
vast complex regions~$\Omega$ are just purely imaginary 
monomials $V(x):=i x^n$ and their perturbations, 
see \eg~\cite{Davies_1999-NSA, Boulton_2002,  
Pravda-Starov_2006, SK, KSTV, Novak-2015-54}. 
Hence, the state of the art of the current research
in construction of the ``large-energy'' pseudomodes 
for (non-semiclassical) Schr\"odinger operators
is by far incomplete and the objective of this paper 
is to fill up the gap. In fact, all known cases 
(as well as all semi-classical ones) represent the simplest 
illustrations of our results, 
see Examples~\ref{ex:pol.1}, \ref{ex:pol.Omega} and \ref{ex:s-c}. 

The fundamental questions that we address here read as follows: 
\begin{itemize}
\item For which potentials does there exist
a non-trivial region $\Omega \subset \Com$ where \eqref{objective.intro} holds?
\item Comparing to $\Im V$,  how large can $\Re V$ be 
so that \eqref{objective.intro} is preserved?
\item Depending on $V$, what is the shape of $\Omega$?
\item Is the polynomial-like character of the so far studied operators important?
\item What is the role of the regularity of $V$? 
\end{itemize}

The main results of this paper giving answers to the raised questions are Theorems \ref{thm:basic} (on pseudomodes for $\lambda \to + \infty$), Theorem \ref{thm:mol} (on pseudomodes for $\lambda \to + \infty$ and potentials of low regularity) and Theorem~\ref{thm:basic.gen} (on pseudomodes for $\lambda \to \infty$ on general curves in $\Com$). These statements proved under technical Assumptions~\ref{asm:basic}, \ref{asm:W} and \ref{asm:basic.gen} are applied to more concrete classes of potentials in Sections~\ref{subsec:Ex.1}, \ref{subsec:Ex.mol} and \ref{subsec:Ex.gen}.

Basically all available results 
on non-trivial pseudospectra of Schr\"odinger operators 
are deduced by scaling 	from \emph{semiclassical pseudomodes},
where a small parameter $h^2$ is added in front of the second-derivative
in~\eqref{objective.intro}, 
see \eg, \cite{Davies_1999-NSA,Dencker-Sjostrand-Zworski_2004}.
However, such an approach has several drawbacks. 
First of all, only very specific (homogeneous or their perturbations) 
potentials can be treated and unboundedness of $\Im V$ 
at infinity may seem to be crucial due to the scaling.
Second, the artificial transition to the new parameter~$h$, 
related in various ways to~$\lambda$, 
complicates the natural interpretation of the results 
as well as the main points in the proofs.
Finally, with the exception of 
the \emph{imaginary shifted harmonic oscillator} 
$V(x):=(x+i)^2$ treated in~\cite{KSTV}, 
no claims seem to be available when $\Re V$ is larger than $\Im V$ at infinity.  
For these reasons, in this paper
we attack the problem directly 
(without introducing the semiclassical parameter~$h$).

The present results also have a connection 
to some open problems posed during the 2015 AIM workshop~\cite{AIM-2015}. 
In particular, we would like to emphasise the following 
insights provided by this paper.

\subsection*{The semiclassical setting as a consequence}
From our approach the known claims in the semiclassical setting 
follow immediately. In particular, the Davies' condition~\cite{Davies_1999-NSA} 
$\Im V'\neq 0$ or its (weaker) versions
(see \cite{Zworski_2001,Pravda-Starov-2004-460}) 
can be easily generalised, see Example~\ref{ex:s-c}.
It is also worth noting that our general non-semiclassical pseudomodes 
do not always localise, instead their support may extend.
 
\subsection*{Optimality of potentials}
Our assumption \eqref{asm.eq:ReV} on the allowed size of $\Re V$ is optimal, 
at least for polynomial-like potentials 
(with $\nu_\pm=-1$ in assumption \eqref{asm.eq:V'}). Indeed, by completely different methods, it has been established in \cite{Mityagin-toappear,Mityagin-2017-272} 
that \eg~for potentials~$V$ 
satisfying $\Re V(x) = |x|^\beta$ with $\beta \geq 1$ and 
\begin{equation}\label{MS.cond}
\exists \epsilon>0, \quad |\Im V(x)|^2 = \BigO(|x|^{\beta-2 - \epsilon}), 
\quad |x| \to \infty,
\end{equation}
the eigensystem of $H_V$ contains a Riesz basis (and there are possibly only finitely many degenerate eigenvalues) and hence the only non-trivial pseudomodes exist for $\lambda$ close to the eigenvalues of $H_V$ (with known asymptotics, see~\cite{Mityagin-toappear}). In turn, the current results suggest that the condition \eqref{MS.cond} is optimal with respect to the Riesz basis property of~$H_V$ (which can be indeed concluded if more information about the position of eigenvalues of~$H_V$ is available) and confirms that the borderline case (potentials with $\epsilon = 0$ in \eqref{MS.cond}) is the most challenging one, see \cite[Open Problem~15.1]{AIM-2015}. Moreover, the assumption \eqref{asm.eq:ReV} has a very natural interpretation, namely, the pseudomodes loose their exponential decay if \eqref{asm.eq:ReV} is not satisfied, see Remark~\ref{rem:ReV.asm}.

\subsection*{Optimality of pseudospectral regions}
Our restrictions on the set $\Omega$ in~\eqref{a.gen}, 
expressed in terms of conditions on $a:=\Re\lambda$ and $b:=\Im\lambda$,
seem to be optimal.
The optimality for the \emph{rotated harmonic oscillator} $V(x):=ix^2$ 
follows by Boulton's conjecture~\cite{Boulton_2002}
solved by Pravda-Starov~\cite{Pravda-Starov_2006},
see Example~\ref{ex:pol.Omega}. 
The lower bound of~\eqref{optimality} is also known to be optimal 
for the \emph{imaginary cubic oscillator} $V(x):=ix^3$, 
see \cite[{Sec.~4.1}]{BordeauxMontrieux-2013}.
The study of optimality of our estimates on the region~$\Omega$
in general cases constitutes an interesting open problem.}
 
\subsection*{Generality}
We are able to treat a wide class of potentials being far beyond polynomial or scalable ones (we also allow a large $\Re V$ without restricting its sign). The method can be further straightforwardly generalised for even wilder potentials than already a quite wide range covered here 
(from bounded or even decaying, see Section~\ref{subsec:dec}, to super-exponential ones). 
For instance, the previously inaccessible (non-scalable) cases like 
$V(x):=i \sinh(x)$ or $V(x):=i \arctan(x)$ are included, 
see Examples~\ref{ex:sinh} and~\ref{ex:arctan}. 
It is also important to stress that 
for realisations in $L^2(\Real)$, 
just the different asymptotic
behaviour of $\Im V$ at $\pm \infty$:
\begin{equation}\label{Ass.intro}
\lim_{x\to -\infty} \Im V(x) 
\ \cdot \
\lim_{x\to +\infty} \Im V(x) 
\ < \ 0 
\end{equation}
(see also \eqref{asm.eq:ImV} for a slight generalisation)
is crucial to ensure the ``significant non-self-adjointness'' of $H_V$ 
and thus the validity of~\eqref{objective.intro} for $\lambda \to + \infty$. 
For decaying but non-integrable potentials~$V$, condition
\eqref{Ass.intro}~can be further weakened 
by requiring that~$\Im V$ approaches~$0$ 
from opposite sides at $\pm \infty$, 
see Section~\ref{subsec:dec}. 
The various conditions of the type  \eqref{Ass.intro} can be viewed as 
a global version of the local Davies' condition $\Im V' \neq 0$ or its weaker versions mentioned above. 

\subsection*{Rough potentials}
In fact, we cover even \emph{discontinuous} potentials, 
which were previously inaccessible to semiclassical techniques.
This is achieved by developing
a robust method of $\lambda$-dependent mollifications
of the potential. 
This new idea enables us to eventually solve an open problem raised during the AIM workshop \cite[Open Problem~10.1]{AIM-2015}.

\subsection*{The regularity of potentials 
and decay rates of pseudomodes}
We explicitly demonstrate the crucial influence of the regularity (or local deformations) of $V$ on the best possible rates in \eqref{objective.intro}. 
The existing results suggest a difference in the rates for analytic and smooth potentials (exponential versus ``faster than any power'' rates), see \eg~\cite{Davies_1999-NSA,Dencker-Sjostrand-Zworski_2004}. 
However, the optimal upper bounds for the resolvent
norm are usually not available and so such observations are not always proved. 
In this paper we stress (and prove) the difference in rates 
for various step-like potentials of the type
($\arctan$ may be replaced by any ``regularisation'' of $\sgn$)
\begin{equation}\label{step-like}
V_1(x):= i \sgn(x) 
\qquad \text{versus} \qquad 
V_2(x):= i \arctan(x).
\end{equation}	
Here the best possible rate is linear in the first case 
(as proved in~\cite{HK} by a careful analysis of 
the resolvent kernel) 
versus the ``faster than any power'' rate in the second case, 
see Example~\ref{ex:arctan}. 
Notice that the even more drastic local deformation, 
namely the operator $ -{\der^2}/{\der x^2} + i \sgn (x)$ 
subject to an \emph{additional} 
Dirichlet boundary condition at~$0$, 
exhibits \emph{no decay} for $\lambda \to +\infty$ in~\eqref{objective.intro}, 
since such an operator becomes normal.

\subsection*{Laptev-Safronov eigenvalue bounds}
Our results for decaying potentials from Section~\ref{subsec:dec} 
show that the bound on individual eigenvalues
of one-di\-men\-sion\-al Schr\"odinger operators 
due to Laptev and Safronov
(see \cite[Thm.~5]{Laptev-Safronov_2009} and \cite[Open Problem 7.1]{AIM-2015})
cannot be improved using the Birman-Schwinger technique (since the norm estimate on the Birman-Schwinger operator provides a resolvent estimate). To justify the latter, we find simple $L^p$-potentials with $p>1$ for which~\eqref{objective.intro} holds, with the decay rate faster than any power of $1/|\lambda|$, in a region~$\Omega$ determined by~\eqref{Omega.dec},
which essentially coincides with the set appearing in \cite[Thm.~5]{Laptev-Safronov_2009}. Thus the very natural reason for the appearance of such $\Omega$ is provided. 

The existence of this region~$\Omega$, 
where the spectrum of $H_V$ is extremely unstable with respect to further, even tiny, perturbations, is a crucial difference with respect to the $L^1$-potentials. 
In the latter case, the resolvent estimate preventing that the resolvent of~$H_V$ explodes for large $\lambda$'s again follows from the Birman-Schwinger estimate, see \cite{Abramov-Aslanyan-Davies_2001}. 

\subsection*{Higher dimensions}
The results and methods of this paper are essentially one-dimensional.
Nonetheless, the results have consequences for multi-dimensional Schr\"odinger operators with (at least local) symmetries and their not too strong perturbations. The pseudomodes in $L^2(\Real)$ from Section~\ref{sec:pseudo.real} are obviously applicable for problems allowing for the separation of variables in Cartesian coordinates, while the pseudomodes in $L^2(\Real_+)$ from Section~\ref{sec:pseudo.gen} are applicable for radially symmetric problems. Finally, the pseudomodes from Example~\ref{ex:sing} arising due to the strongly singular potential at $0$, namely 
\begin{equation}\label{V.sing.intro}
V(r) := \frac{c}{r^2} + \frac{i}{r^\alpha}
\,, \qquad
c \in \Real, \quad \alpha > 2, \quad r>0 
\,,
\end{equation}
localise in a vicinity of $0$ and so are applicable for multi-dimensional potentials
with a local radial singularity of the type \eqref{V.sing.intro}. Unlike in one dimension, these pseudomodes do not show the optimality of region $\Omega$ in Laptev-Safronov multi-dimensional eigenvalue bounds since the condition $\alpha>2$ cannot be satisfied for $V\in L^p(\Real^d)$ with $p \geq d/2$ (or $p>1$ for $d=2$). 
\subsection*{Organisation of the paper}
In Section~\ref{sec:prelim} we outline our strategy 
to construct the pseudomodes and settle a number of important prerequisites
for the subsequent applications.
Section~\ref{sec:pseudo.real} is devoted to large positive pseudoeigenvalues,
while the case of general complex regions is treated only 
in Section~\ref{sec:pseudo.gen};
these two sections are concerned with sufficiently regular
potentials (at least continuous).
Large positive pseudoeigenvalues
for discontinuous and singular potentials
are dealt with in the intermediate Section~\ref{sec:low.reg}.

\subsection*{Notations}
Let us fix some notations employed throughout the paper. 
We use the following conventions for number sets, 
$\Nat := \{1,2,\dots\}$, $\Nat_0:=\Nat \cup \{0\}$, $\Real_+ := (0,\infty)$ and $\Real_- := (-\infty,0)$.
Given an interval $I \subset \Real$, 
the norm of $L^p(I)$ is denoted by $\|\cdot\|_{L^p(I)}$.
If $I=\Real$, we abbreviate $\|\cdot\|_p := \|\cdot\|_{L^p(\Real)}$
and $\|\cdot\| := \|\cdot\|_2$.
The $L^p$ spaces with a weight are denoted by 
\begin{equation}
L^p_\alpha (I) := \{ f \text{ measurable} \, : \, \langle x \rangle^\alpha f(x) \in L^p(I) \}, \quad \alpha \in \Real 
\,,
\end{equation}
where $\langle x \rangle := (1+x^2)^\frac12$.
For an ``integer interval'' we use the double brackets, $[[m,n]]:=[m,n] \cap \Int$. To avoid using many irrelevant constants, we employ the convention that $a \lesssim b$ if there is a constant $C>0$, independent of $\lambda$ and $x$ 
(or any other relevant parameter), such that $a \leq C b$; the convention for $\gtrsim$ is analogous. By $a \approx b$ it is meant that $a \ls b$ and $a \gs b$.

\section{Preliminaries}
\label{sec:prelim}
A standing hypothesis of this paper is that the complex-valued potential~$V$
satisfies the local square-integrability condition
$V \in \sii_\mathrm{loc}(\Real)$.
We understand the Schr\"odinger operator~\eqref{Hamiltonian} 
as the maximal operator generated by the differential expression, 
\ie,
\begin{equation}\label{Hamiltonian}
\begin{aligned}
H_V f &:= -f'' + V f \,,
\\
\Dom(H_V) &:=
\{f \in \sii(\Real) : -f'' + V f \in \sii(\Real) \}
\,.
\end{aligned}
\end{equation}

If $\Re V$ is bounded from below, 
Kato's theorem (\cf~\cite[Sec.~VII.2.2]{Edmunds-Evans}) yields that 
$H_V$ is quasi-m-accretive and, moreover, $\CcR$ is a core of~$H_V$.
The quasi-m-accretivity ensures that~\eqref{Hamiltonian}
is well defined as a closed operator with non-empty resolvent set containing some open left half-plane.
The latter properties of $H_V$ are valid also in the non-accretive case 
under alternative assumptions on~$V$, see \cite{Krejcirik-2016}.
For the pseudomode constructions performed 
in the present paper, however, not even the closedness of~$H_V$ is necessary.
 
\subsection{The JWKB ansatz}
Our construction of pseudomodes
is based on the Liouville-Green approximation 
(also known as the JWKB method), 
see \eg~\cite{Davies_1999-NSA,Olver-1997}.

If~$V$ were constant, \ie\ $V(x) = V_0$ for all $x \in \Real$,
exact solutions of the differential equation
$-g''+V_0 g = \lambda g$ would be given by 
\begin{equation}\label{WKB}
e^{\pm i\int_0^x \sqrt{\lambda-V_0} \, \der t}
\,.
\end{equation}
We shall be particularly interested in the limit $\lambda \to + \infty$
and consistently consider the \emph{principal branch} of the square root.
More generally, we always restrict to 
\begin{equation}\label{contour}
  \lambda \in \Com\setminus (-\infty,0)
  \,.
\end{equation}

For a variable potential~$V$, we still take \eqref{WKB} with $V_0$ replaced by~$V$ and with the minus sign (due to assumptions on the signs of $\Im V$, see \eqref{asm.eq:ImV}) as a basic ansatz to get the approximate solutions~\eqref{objective.intro}. Nonetheless, usually more terms will be needed for unbounded potentials or when~$V$ is sufficiently regular 
and more information on the decay rates in \eqref{objective.intro} 
are sought. 
In general, we therefore take
\begin{equation}\label{g.def}
g(x) := \exp \left(- \sum_{k=-1}^{n-1} \lambda^{-\frac k2} \psi_k(x) \right),
\end{equation}
where functions $\psi_k$ are to be determined. Not surprisingly, $\psi_{-1}$ will turn out to be given by
$\psi_{-1}(x) := i \lambda^{-1/2}\int_0^x \sqrt{\lambda -V(t)} \, \der t$. As we will show in examples  in Section~\ref{subsec:Ex.1}, 
most of interesting potentials 
can be treated already with the expansion~\eqref{g.def} up to $n=2$.

\subsection{The cut-off}
To obtain admissible pseudomodes, 
it is important to employ a $\lambda$-dependent cut-off 
of the JWKB ansatz~\eqref{g.def}.
To this aim, we consider a function $\xi:\Real\to\Real$ satisfying
the following  properties: 
\begin{equation}\label{xi.def}
\begin{aligned}
& \xi \in C_0^\infty(\Real), \quad 0 \leq \xi \leq 1, 
\\
&\forall x \in (-\delta_- +\Delta_-,\delta_+ -\Delta_+), \quad \xi(x) =1,  
\\
\qquad 
& \forall x \notin (-\delta_-,\delta_+), \quad \xi(x)=0; 
\end{aligned}
\end{equation}
the $\lambda$-dependent positive
numbers $\delta_\pm=\delta_\pm(\lambda)$ 
and $\Delta_\pm=\Delta_\pm(\lambda) <\delta_\pm$ 
will be determined later. Notice that $\xi$ can be selected in such a way that
\begin{equation}\label{xi.der}
\|\xi^{(j)}\|_{L^\infty(\Real_\pm)} \ls \Delta_\pm^{-j}, \quad j=1,2 .
\end{equation}
To simplify notations, we also define intervals
\begin{equation}\label{J.def}
\begin{aligned}
\cJ&:=
(-\delta_-,\delta_+),& &\cJ_\pm :=\{x \in \Real_\pm \, : \, |x| < \delta_\pm \},
\\
\cJ'&:=
(-\delta_- + \Delta_-,\delta_+ -\Delta_+), && \cJ_\pm':=\{x \in \Real_\pm \, : \, |x| < \delta_\pm-\Delta_\pm \}.
\end{aligned}
\end{equation}
Our ansatz for a general potential~$V$ then reads 
\begin{equation}\label{Ansatz}
f := \xi \,  g ,
\end{equation}
where~$g$ is defined in~\eqref{g.def} and the index $n \in \Nat_0$
will be chosen according to the smoothness of $V$.

\subsection{The strategy}
Let us informally describe the strategy. 
Recalling~\eqref{g.def}, we have
\begin{equation}\label{triangle}
\begin{aligned}
-f'' + (V-\lambda)f 
&= -(\xi \, g)'' + (V-\lambda)  \xi \, g
\\
&= -\xi'' g - 2 \xi' g' 
+ \xi [-g'' + (V-\lambda) g].
\end{aligned}
\end{equation}
When $n=0$, the appearing terms read
\begin{align}
g' 
= - i \sqrt{\lambda-V} g,
\qquad 
-g'' + (V-\lambda) g
= \frac{- iV'}{\sqrt{\lambda-V}} \, g
\,,
\end{align} 
which already suggests what needs to be done. First, $V$ must be sufficiently regular so that $f \in \Dom(H_V)$; in fact, the more terms in \eqref{g.def} are taken, the more regular $V$ is needed. Second, the functions $\psi_k$ in \eqref{g.def} and the cut-off $\xi$ must be selected in such a way that 
the $L^2$-norm of the third term on the second line of \eqref{triangle} 
is as small as possible when $\lambda$ is large. 
Third, the assumption on the sign of $\Im V$, 
see~\eqref{asm.eq:ImV}, implies that $|g|$ decays exponentially, see Lemma~\ref{lem:g.est}, and so the terms with $\xi'$ and $\xi''$ are expected to be small; nevertheless, an appropriate restriction of $\delta_\pm$, $\Delta_\pm$ must be given.

Since our goal is to deal with potentials of low regularity, the construction consists of more steps. 
First we deal with sufficiently regular potentials $V$, later we add a singular term $W$ and follow various possible strategies how to treat it, see Section~\ref{sec:low.reg}. 

\subsection{The expansion}

For~$g$ given in~\eqref{g.def}, we have 
\begin{equation}\label{H.g.reg}
\begin{aligned}
-g'' + (V-\lambda)g &= 
\left( 
\sum_{k=-1}^{n-1} \lambda^{-\frac k2} \psi_k'' 
\right) g - 
\left( \sum_{k=-1}^{n-1} \lambda^{-\frac k2} \psi_k' 
\right)^2 g + (V - \lambda) g
\\
&=: \left( \sum_{k=-2}^{2(n-1)} \lambda^{-\frac{k}{2}} \phi_{k+1} \right) g,
\qquad n \in \Nat.
\end{aligned}  
\end{equation}
Here the functions~$\phi_k$ with $k \in [[-1,2n-1]] $ are naturally defined 
after grouping together the terms with the same power of~$\lambda$
on the right hand side of the first line in~\eqref{H.g.reg},
with the exception of~$V$ which we include in the leading order term:
\begin{equation}\label{scheme}
\begin{aligned}
&(k=-2) 
&\lambda^1: 
&  &- (\psi_{-1}')^2 + \frac{V - \lambda}{\lambda}  &=: \phi_{-1}, 
\\
&(k=-1) \ 
&\lambda^\frac12: &   &\psi_{-1}'' - 2\psi_{-1}' \psi_0' &=: \phi_0, 
\\ 
&(k=0) \ 
&\lambda^0: &  & \psi_0'' - 2\psi_{-1}' \psi_1' - (\psi_0')^2 &=: \phi_1,
\\
&(k=1) \ 
&\lambda^{-\frac 12}: & &   \psi_1'' - 2\psi_{-1}' \psi_2' - 2\psi_{0}' \psi_1' &= : \phi_2,
\\
&\dots
\end{aligned}
\end{equation}
For $-1 \leq k \leq 2(n-1)$, the formulae can be written concisely as  
\begin{equation}\label{binomial}
\psi_k'' - \sum_{\alpha+\beta=k} \psi_\alpha' \psi_\beta' = \phi_{k+1}
\,,
\end{equation}
with the convention that $\psi_\alpha=0$ 
whenever $\alpha \geq n$ or $\alpha \leq -2$.

For the given $n \in \Nat$, 
we have $n+1$ functions $\psi_{-1},\dots,\psi_{n-1}$ 
and $2n+1$ functions $\phi_{-1},\dots,\phi_{2n-1}$. 
The strategy is to require that the first $n+1$ functions
$\phi_{-1},\dots,\phi_{n-1}$
are equal to zero, which determines all available $\psi_k$.
Using~\eqref{binomial},
this leads to a system of $n+1$ first-order differential equations
that the functions $\psi_{-1},\dots,\psi_{n-1}$ must satisfy:
\begin{equation}\label{ODE}
\begin{aligned}
\psi_{-1}' 
&=  i \lambda^{-\frac 12} (\lambda - V)^\frac12 
\,,
\\
\psi_{k+1}' 
&= \frac{1}{2\psi_{-1}'}
\Bigg(
\psi_k'' - \sum_{\stackrel[\alpha,\beta \not =-1]{}{\alpha+\beta=k}}
\psi_\alpha' \psi_\beta' 
\Bigg)
\,, \qquad
k \in [[-1,n-2]] 
\,,
\end{aligned}
\end{equation}
with the convention as above that $\psi_\alpha=0$ 
whenever $\alpha \geq n$ or $\alpha \leq -2$.
Here and in the sequel~$\lambda$ is, in addition to~\eqref{contour},
assumed to be such that $\lambda-V(x) \in \Com\setminus(-\infty,0)$
for all $x \in \Real$. 
Recall that the principal branch of the square root is considered in this paper.

Notice that we were free to choose the sign 
in the definition of~$\psi_{-1}'$ to make $\phi_{-1}=0$, see~\eqref{scheme}.
Our choice made in~\eqref{ODE} will be consistently
followed in this paper.

Finally, with this choice of functions $\psi_k$ we get 
\begin{equation}\label{rem.phi}
-g'' + (V-\lambda)g
= 
\left(
\sum_{k=n-1}^{2(n-1)} \lambda^{-\frac{k}{2}} \phi_{k+1}
\right) g =:r_n \, g, \qquad n \in \Nat.
\end{equation}

The essential point for estimating the resulting term is the understanding of the structure of functions $\psi_k'$ and remainders $r_n$, which is the content of the following lemmata. The proof is based on a straightforward but rather lengthy induction argument, see Appendix.

\begin{Lemma}\label{lem:str}
Let $n \in \Nat_0$, $V \in W^{n+1,2}_{\rm loc}(\Real)$ and functions $\{\psi'_k\}_{k\in[[-1,n-1]]}$ be determined by \eqref{ODE}.
Then
\begin{equation}\label{psi.k.m}
\psi_k^{(m)} = \frac{\lambda^\frac k2}{(\lambda-V)^\frac k2} \sum_{j=0}^{k+m} \frac{T_j^{k+m,k+m+1-j}}{(\lambda-V)^j},  \quad m\in [[1,n+1-k]],
\end{equation}
where (with some $c_\alpha \in \Com$)
\begin{equation}\label{Tjr.def}
\begin{aligned}
T_j^{r,s} & := 
\sum_{\alpha \in \cI_j^{r,s}} c_\alpha (V^{(1)})^{\alpha_1} (V^{(2)})^{\alpha_2} \dots (V^{(s)})^{\alpha_s},
\\
\cI_j^{r,s} & := 
\left\{ \alpha \in \Nat_0^s \, : \, \sum_{i=1}^s i \alpha_i = r \ \&  \ \sum_{i=1}^s  \alpha_i = j \right\}. 
\end{aligned}
\end{equation}
Moreover, if $r \geq 1$, then $\cI_0^{r,r+1} = \emptyset$.
\end{Lemma}
\begin{Lemma}\label{lem:r_n.est}
Let $n \in \Nat_0$, $V \in W^{n+1,2}_{\rm loc}(\Real)$ and functions $\{\psi'_k\}_{k\in[[-1,n-1]]}$ be determined by \eqref{ODE}, $\{\phi_k\}_{k \in [[-1,2n-1]]}$ be as in \eqref{binomial} and $r_n$ as in \eqref{rem.phi}.
Then
\begin{equation}\label{rn.est}
|r_n| \ls 
\frac{|V^{(n+1)}|}{|\lambda-V|^\frac{n+1}{2}}
+
\sum_{k=0}^{n-1} \frac{1}{|\lambda-V|^\frac{n-1+k}2} \sum_{l=2}^{n+1+k} \frac{|T_l^{n+1+k,n}|}{|\lambda-V|^l},
\end{equation}
where $T_j^{r,s}$ are as in \eqref{Tjr.def}.
\end{Lemma}

As an illustration for the expansions above with $n=0,1,2$ 
we list functions $\psi_k'$: 
\begin{equation}\label{psi'.list}
\begin{aligned}
(n=0) &\qquad&
\psi_{-1}' &= i \frac{(\lambda - V)^\frac12}{\lambda^\frac 12}, 
\\
(n=1) &&
\psi_0' & = -\frac14 \frac{V'}{\lambda - V}, 
\\
(n=2) &&
\psi_1' &= \frac{i}{2} 
\frac{\lambda^\frac 12}{(\lambda - V)^\frac12} 
\left( \frac14 \frac{V''}{\lambda-V} 
+ \frac{5}{16} \frac{V'^2}{(\lambda-V)^2}  \right),
\end{aligned}
\end{equation}
together with 
the remainders $r_n$ on the right of~\eqref{rem.phi}: 
\begin{equation}\label{rn.est.list}
\begin{aligned}
r_0 & = -\frac{i}{2} \frac{V'}{(\lambda-V)^\frac12},
\\
r_1 & = -\frac{1}{4} \frac{V''}{\lambda-V} 
- \frac{5}{16} \frac{V'^2}{(\lambda-V)^2},
\\
r_2 & = \frac{i}{8} \frac 1{(\lambda-V)^\frac12}  
\left( \frac{V'''}{(\lambda-V)} 
+ \frac{9}{2} \frac{V'V''}{(\lambda-V)^2} 
+ \frac{15}{4} \frac{V'^3}{(\lambda-V)^3}\right) 
\\ & \quad + \frac{1}{64} \frac 1{(\lambda-V)}  \left( 
\frac{V''^2}{(\lambda-V)^2} 
+ \frac{5}{2} \frac{V'^2 V''}{(\lambda-V)^3} 
+ \frac{25}{16} \frac{V'^4}{(\lambda-V)^4}
\right).
\end{aligned}
\end{equation}

\section{Pseudomodes for $\lambda \to + \infty$}
\label{sec:pseudo.real}

In this section, unless otherwise stated,
we always assume that~$\lambda$ is \emph{positive}
and typically very large.

\subsection{Admissible class of potentials}
We proceed under the following hypothesis about the (possibly unbounded) potential~$V$.

\begin{Assumption}\label{asm:basic}
Let $N \in \Nat$  and let $V \in W^{N,\infty}_{\rm loc}(\Real)$ satisfy the following conditions:
\begin{enumerate}[a)]
\item $\Im V$ has a different asymptotic behaviour at $\pm \infty$:
\begin{equation}\label{asm.eq:ImV}
\limsup_{x\to -\infty} \Im V(x) <0 , \qquad  \liminf_{x\to +\infty} \Im V(x) >0;
\end{equation}
\item derivatives of $V$ are controlled by $V$: $\exists \nu_\pm \in \Real,$  $\forall m \in [[1,N]]$,
\begin{equation}\label{asm.eq:V'}
\begin{aligned}
|\Im V^{(m)}(x)| &= \BigO \left(|\Im V(x)| \langle x \rangle^{m\nu_\pm} \right), \quad x \to \pm \infty,
\\
|V^{(m)}(x)| &= \BigO \left(|V(x)| \langle x \rangle^{m\nu_\pm} \right), \quad x \to \pm \infty;
\end{aligned}
\end{equation}
\item $\Im V$ is sufficiently large:
\label{asm:basic.c}
\begin{enumerate}[i)]
	\item if $V$ is unbounded at $\pm \infty$, then suppose that: $\exists \eps_1 > 0$, 
\begin{equation}\label{asm.eq:ReV}
\begin{aligned}
\langle x \rangle^{4(\nu_\pm +\eps_1)+2} 
(\langle x \rangle^{4\nu_\pm +2} + |\Re V(x)| )  
= \BigO\left(|\Im V(x)|^2\right), \quad x \to \pm \infty;
\end{aligned}
\end{equation}
\item if $V$ is bounded at $\pm \infty$, 
then suppose that $\nu_\pm<1$, 
where $\nu_\pm$ are the numbers from \eqref{asm.eq:V'}.
\end{enumerate}
\end{enumerate}
\end{Assumption}

\begin{Example}
A model case of $V$ satisfying Assumption~\ref{asm:basic}
is a sufficiently regular polynomial-like potential satisfying
\begin{equation}
V(x) = |x|^\beta + i \sgn x |x|^\gamma, \qquad |x| \geq 1,
\end{equation}
for which $\nu_\pm = -1$, \cf~\eqref{asm.eq:V'}. The conditions \eqref{asm.eq:ImV} and \eqref{asm.eq:ReV} hold if $\gamma \geq 0$ and $\gamma >(\beta-2)/2$, respectively.
More examples can be found in Section~\ref{subsec:Ex.1} and the optimality of the conditions in Assumption~\ref{asm:basic} is discussed around \eqref{MS.cond}.
\end{Example}

Several comments on the assumption are in place. 
First of all, notice that~$V$ is necessarily continuous
due to $V \in W^{1,\infty}_{\rm loc}(\Real)$.

The condition \eqref{asm.eq:ImV} ensures that the operator~\eqref{Hamiltonian.intro} is ``significantly non-self-adjoint''.
More precisely, $H_V$~is not normal as a consequence of~\eqref{asm.eq:ImV},
the normality is equivalent to the condition that $\Im V$ is constant. 
Furthermore, hypothesis~\eqref{asm.eq:ImV} ensures that
the pseudomode~$g$ from \eqref{g.def} is exponentially decaying. The correct sign for the decay can be seen (if  the shape of $g$ is determined mainly by $\psi_{-1}$) by employing \eqref{asm.eq:ImV} and the complex square root formula
\begin{equation}\label{sqr.rt}
\begin{aligned}
\Re \left( \lambda^{\frac 12} \psi_{-1}' \right)
&= 
- \Im(\lambda-V)^\frac{1}{2} 
\\ &=
\frac{1}{ 2^\frac12} \frac{\Im V}
	{\left(\left[(\lambda-\Re V)^2+(\Im V)^2\right]^\frac12+(\lambda-\Re V)\right)^\frac 12} \,, 
\end{aligned}
\end{equation}
valid for all positive~$\lambda$ satisfying in addition  
the requirement $\lambda -V(x) \in \Com\setminus(-\infty,0)$
for all $x \in \Real$.
The remaining two intertwined conditions guarantee 
that the exponential decay of~$g$ is not spoiled by too large $\Re V$ 
or too wild behaviour of the derivatives of~$V$. 

The condition \eqref{asm.eq:V'} restricts the growth and oscillations of $V$, nonetheless, it is still quite flexible as 
\eg~$V(x) := \pm i e^{x^2}$ for $x \to \pm \infty$ is covered. 
Notice that Gronwall's inequality implies that (with some $M>0$)
\begin{equation}\label{Gronw.est}
\forall x \geq 0, \quad |V(x)| \ls
\left \{
\begin{aligned}
& e^{M x^{1+\nu_+}},
& \nu_+ > -1,
\\
&\langle x \rangle^M, & \nu_+  =-1,
\\
& 1, & \nu_+ < -1;
\end{aligned} 
\right.
\end{equation} 
an analogous estimate holds also for $x \leq 0$. 

If~$V$ is bounded, we do not require that the derivatives of~$V$ are bounded
in general, thus \eg~rapidly oscillating potentials are allowed. 
In such cases, the estimate from Gronwall's inequality becomes very crude. 

The condition \eqref{asm.eq:V'} also implies that for  $\nu_+ \geq -1$ and all sufficiently large $x>0$ and every $|\Delta| \leq x^{-\nu_+}/4$, 
we have
\begin{equation}\label{V.x.Delta}
\frac{|\Im V(x+2\Delta)|}{|\Im V(x)|} \approx 1.
\end{equation}
Indeed, for $\nu_+>-1$, 
\begin{equation}
\begin{aligned}
\left|\log \frac{|\Im V(x+2\Delta)|}{|\Im V(x)|}\right| 
&= 
\left|\int_x^{x+2\Delta} \frac{\Im V'(t)}{\Im V(t)} \, \der t \right|
\ls
\left| |x+2\Delta|^{\nu_+ +1}-|x|^{\nu_+ +1} \right|
\\
&\ls 
x^{\nu_+} |\Delta|  + \BigO(|\Delta|^2 x^{\nu_+-1}),
\end{aligned}
\end{equation}
and similarly for $\nu_+=-1$.
The conclusion \eqref{V.x.Delta} is clearly valid also for a bounded $V$ as we require \eqref{asm.eq:ImV}.

\subsection{Localisation of the pseudomode and cut-off}
\label{subsec:cut-off}

To estimate $|g|$ we first show that under Assumption~\ref{asm:basic} 
the function 
$\int_0^x [\lambda^\frac12 \psi_{-1}'(t) + \psi'_0(t)] \, \der t$
in the expansion~\eqref{g.def}
dominates over the other terms with $k\geq 1$. 
Thus estimates simplify significantly even for
many terms in~\eqref{g.def}. 
Already at this step, it is important to employ a suitable cut-off. Namely, for every $\lambda >0$ we define:
\begin{equation}\label{delta.def}
\begin{aligned}
\delta_\pm &:= 
\begin{cases}
\displaystyle
\inf \left\{ \delta \geq 0 \, : \, 
\frac{|\Im V(\pm \delta)|^2}{\langle \delta \rangle^{4 \nu_\pm + 2 \eps_1+2}} = \lambda    \right\} & \text{if $V$ is unbounded at $\pm \infty$},
\\[4mm]
\displaystyle
\lambda^{\frac{1+\eps_2}2} \quad \text{with } \quad 0<\eps_2< 1-\nu_\pm & \text{if $V$ is bounded at $\pm \infty$},
\end{cases}
\\
\Delta_\pm &:= \frac 14
\begin{cases}
\displaystyle
\delta_\pm^{-\nu_\pm}  & \text{if $V$ is unbounded at $\pm \infty$} ,
\\[1mm]
\displaystyle
\delta_\pm  & \text{if $V$ is bounded at $\pm \infty$}.
\end{cases}
 \end{aligned}
\end{equation}
\begin{Remark}[Properties of $\delta_\pm$ and $\Delta_\pm$]
	\label{rem:delta}
The following hold.
\begin{enumerate}[i)]
\item 
The infimum can be infinite 
($\inf \emptyset = + \infty$), however, 
for all sufficiently large $\lambda>0$, 
the numbers $\delta_\pm$ are finite and 
\begin{equation}\label{delta.lim}
\lim_{\lambda\to + \infty} \delta_\pm = + \infty ;
\end{equation}
	\item $\Delta_\pm$ are so small that the values of $\Im V(x)$ are comparable if $|x - \delta_\pm| \leq 2\Delta_\pm$;
	\item for all sufficiently large  $\lambda>0$,
	\begin{equation}\label{lam.V.comp}
	\forall x \in \cJ, \quad |\lambda-V(x)| \approx \lambda.
	\end{equation}
\end{enumerate}
\end{Remark}
\begin{proof} 
All the three properties are obvious for bounded~$V$. 
We further assume that $V$ is unbounded at $+\infty$ 
and prove the claims; 
the case of $V$ unbounded at $-\infty$ is analogous.	

i) It follows from the assumption~\eqref{asm.eq:ReV} that for all sufficiently large $\delta>0$
\begin{equation}
\frac{|\Im V(\delta)|^2}{\langle \delta \rangle^{4 \nu_+ + 2 \eps_1+2}} 
\gs
\langle \delta \rangle ^{2 \eps_1},
\end{equation}
thus for all 
\begin{equation}
\lambda > \min_{\delta \geq 0 } 
\frac{|\Im V(\delta)|^2}{\langle \delta \rangle^{4 \nu_+ + 2 \eps_1+2}}  
\end{equation}
the number~$\delta_+$ is finite. 
Since $\Im V$ is continuous, \eqref{delta.lim} follows.

ii) See \eqref{V.x.Delta}.

iii) From \eqref{asm.eq:ReV}, 
we obtain that for all $x>x_0$ with $x_0$ sufficiently large,  
\begin{equation}
|\Im V(x)| 
\ls  \frac{|\Im V(x)|^2}{\langle x\rangle ^{4 \nu_+ + 2 \eps_1+2}},
\qquad 
|\Re V(x)| 
\ls  \frac{|\Im V(x)|^2}{\langle x\rangle^{4 \nu_+ + 4 \eps_1+2}};
\end{equation}
thus we may assume that $x_0$ is chosen so large that for all $x>x_0$, we have 
\begin{equation}
|\Re V(x)| 
\leq  \frac 12 \frac{|\Im V(x)|^2}{\langle x\rangle^{4 \nu_+ + 2 \eps_1+2}}.
\end{equation}
Thus, using \eqref{delta.def}, already proved i) and the continuity of $V$, we can select sufficiently large $\lambda_0>0$ such that for all $\lambda>\lambda_0$ and all $x \in [0,\delta_+]$, we have
\begin{equation}\label{lambdapul}
|\Im V(x)| 
\ls  \lambda,
\qquad 
|\Re V(x)| 
\leq  \frac \lambda 2.
\end{equation}
Hence \eqref{lam.V.comp} follows.
\end{proof}
\begin{Remark}[More on the assumption on $\Re V$]
\label{rem:ReV.asm}
The restriction on $\Re V$ made in~\eqref{asm.eq:ReV}
arises in a very natural way. 
As an illustration, let us consider the potential
$V (x): = |x|^\beta + i \sgn (x) \, |x|^\gamma$
with positive powers~$\beta,\gamma$.
In this case we can take $\nu_\pm:=-1$ 
to satisfy hypothesis~\eqref{asm.eq:V'}
and assumption~\eqref{asm.eq:ReV} imposes
the relationship $\beta-2+4\eps_1 \leq 2\gamma$.
Choosing on the contrary $\beta > 2\gamma +2$
so that~\eqref{asm.eq:ReV} is violated,
the dominant part in the expansion~\eqref{g.def}, 
which is given by~$\Re \left( \lambda^{\frac 12} \psi_{-1}' \right)$ 
as we show later, becomes (uniformly in $\lambda$) 
\emph{integrable} in $x\in \Real$. 
Indeed, with the substitution 
$t = \lambda^{1/\beta}s$ and by  straightforward estimates,
\begin{equation}\label{ReV.ex}
\int_0^{\delta_+} 
\left|\Re \left( \lambda^{\frac 12} \psi_{-1}'(t) \right)\right| \, \der t
\ls 
\lambda^{\frac{2+2\gamma - \beta}{2\beta}}
\int_0^\infty
\frac{s^\gamma}{|1-s^\beta|^\frac 12} \, \der s 
= o(1), \quad \lambda \to +\infty.
\end{equation}
Moreover, notice that by taking a larger $\delta$, so that it is possible that $\Re V(x) \geq \lambda$, 
problems appear in the control of the decay of $|r_n|$
for which the estimate $|\Re V - \lambda| \gs \lambda$ is essential.	
\end{Remark}

Using the properties of~$\delta_\pm$ and~$\Delta_\pm$,
we obtain the following estimates.
\begin{Lemma}\label{lem:exp.bound}
Let Assumption~\ref{asm:basic} hold, 
let $0 \leq n \leq N$ and $\{\psi'_k\}_{k\in[[-1,n-1]]}$ be determined by \eqref{ODE} 
and let $\delta_\pm$ be as in \eqref{delta.def}. 
Then for all sufficiently large $\lambda >0$  
\begin{equation}
\label{Re.psi-1}
\forall x \in \cJ, \qquad \lambda^\frac 12 \Re\psi_{-1}'(x) 
\approx \lambda^{-\frac12} \Im V(x),
\end{equation}
and 
\begin{equation}\label{|psi_k|}
\forall k \in [[0,n-1]], \quad \forall x \in \cJ_\pm, \qquad
\lambda^{-\frac k2} |\psi_k'(x)| 
\ls 
\frac{|V(x)|\langle x \rangle^{(k+1)\nu_\pm}}{\lambda^{\frac k2+1}} .
\end{equation}
\end{Lemma}
\begin{proof}
The estimate \eqref{Re.psi-1} follows immediately from \eqref{sqr.rt} 
using~\eqref{lam.V.comp} and~\eqref{lambdapul}.
The rest is based on Lemma~\ref{lem:str} and assumptions \eqref{asm.eq:V'}, \eqref{asm.eq:ReV}. 

For $k \geq 0$, Lemma~\ref{lem:str} yields
\begin{equation}\label{psik'.exp}
\begin{aligned}
\lambda^{-\frac k2} |\psi_k'|
&\leq 
\frac{1}{|\lambda-V|^\frac k2} \sum_{j=1}^{k+1} \frac{|T_j^{k+1,k+2-j}|}{|\lambda-V|^j}
\\
&\ls 
\frac{1}{|\lambda-V|^\frac k2}
\sum_{j=1}^{k+1}
\frac{\sum_{\alpha \in \cI_j^{k+1,k+2-j}}  
|V'|^{\alpha_1} |V''|^{\alpha_2} \dots |V^{(k+2-j)}|^{\alpha_{k+2-j}}}{|\lambda-V|^j}.
\end{aligned}
\end{equation}
Notice that the highest appearing derivative of $V$ is $V^{(n)}$ and that the product of $|V^{(i)}|$ always consists of $j$ factors (counting with powers) since $\sum_{i=1}^{k+2-j} \alpha_i = j$; see \eqref{Tjr.def}. 
Thus all appearing derivatives of $V$ are continuous and the assumption \eqref{asm.eq:V'} with $\sum_{i=1}^{k+2-j} i \alpha_i = k+1$ from \eqref{Tjr.def} yields that for all sufficiently large $x>0$
\begin{equation}\label{psik'.est}
\lambda^{-\frac k2} |\psi_k'(x)|
\ls
\frac{\langle x \rangle^{(k+1)\nu_+}}{|\lambda-V(x)|^\frac k2} 
\sum_{j=1}^{k+1} \frac{|V(x)|^j }{|\lambda-V(x)|^j}
\ls
\frac{|V(x)| \langle x \rangle^{(k+1)\nu_+}}{\lambda^{\frac k2+1}};
\end{equation}
in the last step we have used~\eqref{lam.V.comp}.
For small $x>0$, the estimate \eqref{|psi_k|} follows immediately from \eqref{psik'.exp} and the continuity of the derivatives of $V$. For $x<0$, the estimate is analogous.
\end{proof}

Localising the ansatz~\eqref{g.def} on the interval~$\mathcal{J}$,
the preceding lemma shows that the shape of~$g$ 
is determined mainly by $\psi_{-1}$ and $\psi_0$.
More specifically, given the derivatives~$\psi_k'$ from~\eqref{ODE},
henceforth we choose the primitive functions~$\psi_k$ by fixing the integration
constant by the requirement
\begin{equation}\label{constant.fixing}
  \psi_k(0) := 0 
  \,, \qquad
  k \in [[-1,n-1]]
  \,.
\end{equation}
With this standing convention, 
we have the following two-sided bounds.
\begin{Lemma}\label{lem:g.est}
Let Assumption~\ref{asm:basic} hold, 
$g$ be as in \eqref{g.def} with $\{\psi'_k\}_{k\in[[-1,n-1]]}$, 
$0 \leq n \leq N$, determined by~\eqref{ODE} and let $\delta_\pm$ 
be as in \eqref{delta.def}. 
Then there exist $c_1, c_2 >0$ such that 
for all sufficiently large $\lambda >0$ and all $x \in \cJ$ we have
\begin{equation}\label{|g|.est}
\exp\left(-\frac{c_1}{\lambda^\frac12}  \int_0^{|x|} |\Im V(t)| \, \der t \right)  
\ls |g(x)| \ls  
\exp\left(-\frac{c_2}{\lambda^\frac12}  \int_0^{|x|} |\Im V(t)| \, \der t \right).
\end{equation}
\end{Lemma}
\begin{proof}
Notice that the formula \eqref{psi'.list} for $\psi_0'$ is exceptional since it can be easily integrated, hence
\begin{equation}\label{g.formula}
g(x) = \frac{[\lambda-V(0)]^\frac 14}{[\lambda-V(x)]^\frac 14} \,
\exp \left( 
- \sum_{\stackrel[k \neq 0]{}{k = -1}}^{n-1} 
\lambda^{-\frac k2} \int_0^x \psi_k'(t) \, \der t 
\right).
\end{equation}
From \eqref{lam.V.comp} we get
\begin{equation}
\forall x \in \cJ, \quad \left|\frac{\lambda-V(0)}{\lambda-V(x)}\right| \approx 1.
\end{equation}
We continue with estimates for $x>0$, the other case is analogous. For any $x_0>0$ fixed, we have from Lemma~\ref{lem:exp.bound} that
\begin{equation}
\left|\Re \sum_{k=-1}^{n-1} \lambda^{-\frac k2} \int_0^{x_0} \psi_k'(t) \; \der t\right| \ls \lambda^{-\frac 12}.
\end{equation}
The remaining estimate for $x>x_0$ follows from \eqref{|psi_k|}, \eqref{delta.def} and assumption \eqref{asm.eq:ReV}, namely
\begin{equation}
\Re \sum_{\stackrel[k \neq 0]{}{k = -1}}^{n-1} \lambda^{-\frac k2} \int_{x_0}^x \psi_k'(t) \; \der t 
= \int_{x_0}^{x} \lambda^{\frac 12} \Re \psi_{-1}'(t) \, [1 + S(t)] \; \der t ,
\end{equation}
where
\begin{equation}
|S(t)| \ls 
\begin{cases}
\lambda^{-\frac 12}& \text{if } V \text{ is unbounded},
\\
\lambda^{-\frac 12}& \text{if } V \text{ is bounded and } \nu_+ < 0,
\\[1mm]
\lambda^{-\frac{1-(1+\eps_2) \nu_+ }2}& \text{if } V \text{ is bounded and } \nu_+ \geq 0.
\end{cases} 
\end{equation}
Indeed, in the first case, Lemma~\ref{lem:exp.bound}, assumption~\eqref{asm.eq:ReV}, \eqref{lam.V.comp} and \eqref{Re.psi-1}  give
\begin{equation}
\frac{\lambda^\frac k2|\psi_k'(x)|}{\lambda^\frac12 |\Re \psi_{-1}'(x)|} 
\ls 
\frac{\langle x \rangle^{(k+1)\nu_\pm}}{\lambda^{\frac {k-1}2 } |\Im V(x)|} 
\ls
\frac{1}{\langle x \rangle^{\frac{k+1}2(\eps_1+1)} \lambda^{\frac {k+1}4 }};
\end{equation}
the other cases can be verified similarly.

Hence using \eqref{Re.psi-1} we get 
\begin{equation}
\Re \sum_{\stackrel[k \neq 0]{}{k = -1}}^{n-1} 
\lambda^{-\frac k2} \int_{x_0}^x \psi_k'(t) \; \der t 
\approx 
\lambda^{-\frac 12} \int_{x_0}^{x}  \Im V(t) \; \der t.
\end{equation}
Putting all estimates from above together, we obtain \eqref{|g|.est}.
\end{proof}

The following proposition ensures that the terms in~\eqref{triangle}
containing derivatives of the cut-off function $\xi$ are negligible in a sense.
\begin{Proposition}\label{prop:cut}
Let Assumption~\ref{asm:basic} hold, $g$ be as in \eqref{g.def} with $\{\psi'_k\}_{k\in[[-1,n-1]]}$, $0 \leq n \leq N$, determined by \eqref{ODE} and $\xi$ be as in \eqref{xi.def} with $\delta_\pm,$ $\Delta_\pm$ as in \eqref{delta.def}. 
Then
\begin{equation}\label{cut.rates}
\kappa(\lambda) :=
\frac{\|\xi'' g\| + \|\xi' g'\|}{\|\xi g\|} 
= o(1), 
\qquad \lambda \to + \infty.
\end{equation}
More precisely, $\kappa(\lambda) = \kappa_-(\lambda) + \kappa_+(\lambda)$ where (with some $c>0$) 
\begin{equation}\label{kappa.est}
\kappa_\pm(\lambda) = 
\begin{cases}
	\BigO\left(\exp(-c \, \delta_\pm^{\nu_\pm + 1+\eps_1}) \right) & \text{if } V \text{ is unbounded at } \pm \infty,
	\\[2mm]
	\BigO\left(\exp \left(-c \, \lambda^{\frac{\eps_2}{2}} \right) \right)& \text{if } V \text{ is bounded at } \pm \infty. 
\end{cases}
\end{equation}

\end{Proposition}
\begin{proof}
First notice that we have $\|\xi g\| \gs 1$ from \eqref{|g|.est}. The main step is to estimate $|g(x)|^2$ for 
$x \in \overline{\cJ \setminus \cJ'}$ where $\xi'$ and $\xi''$ are supported.
We give details only for $x>0$; the other case is analogous.

We start with the case when $V$ is unbounded at $+\infty$. Let $x_0>0$ be so large that $\Im V(x) >0$ for all $x > x_0$. From the property \eqref{V.x.Delta} and the selected size of $\Delta_+$, see \eqref{delta.def},  we obtain for $x \in \cJ_+ \setminus \cJ_+'$ that
\begin{equation}
\int_{x_0}^{x} \Im V(t) \, \der t 
 \geq  
\int_{\delta_+-2\Delta_+}^{x} \Im V(t) \, \der t
\gs
\Delta_+ \Im V(\delta_+)  
\gs 
\frac{\Im V(\delta_+)}{\delta_+^{\nu_+}} .  
\end{equation}
Thus using \eqref{delta.def}, we get
\begin{equation}
\begin{aligned}
\lambda^{-\frac 12}\int_0^x \Im V(t) \, \der t 
&=
\lambda^{-\frac 12}\int_0^{x_0} \Im V(t) \, \der t  + \lambda^{-\frac 12}\int_{x_0}^{x} |\Im V(t)| \, \der t 
\\
& \gs 
-\lambda^{-\frac 12}+
\frac{\delta_+^{2\nu_+ +\eps_1+1}}{\Im V(\delta_+)} \frac{\Im V(\delta_+)}{\delta_+^{\nu_+}}
\gs \delta_+^{\nu_+ +\eps_1+1}.
\end{aligned}
\end{equation}
Hence it follows from \eqref{|g|.est} that (with some $c_3>0$)
\begin{equation}
\forall x \in \cJ_+ \setminus \cJ_+', \qquad 
|g(x)| \ls \exp(- c_3 \delta_+^{\nu_+ + \eps_1+1}).
\end{equation}

Additional terms appearing in $\|\xi' g'\|_{L^2(\Real_+)}$ can be estimated using \eqref{xi.der}, \eqref{delta.def}, \eqref{|psi_k|}, \eqref{lam.V.comp} and \eqref{Gronw.est}. In detail, for all $x \in \cJ_+ \setminus \cJ_+'$ we have (with some $c_4>0$) 
\begin{equation}
\begin{aligned}
|\xi'(x) g'(x)| &
\ls  
\delta_+^{\nu_+} 
\sum_{k=-1}^n \lambda^\frac k2|\psi_k'(x)|
\exp(- c_3 \delta_+^{\nu_+ + \eps_1+1})
\\
&\ls \delta_+^{\nu_+} \left(\lambda^\frac12 + \sum_{k=0}^n \frac{\langle x \rangle^{(k+1)\nu_+}}{\lambda^\frac k2} \right) \exp(- c_3 \delta_+^{\nu_+ + \eps_1+1})
\\
&\ls \delta_+^{\nu_+}  \left( \frac{|V(\delta_+)|}{\delta_+^{2\nu_+ +\eps_1 +1}}  + \delta_+^{(n+1)\nu_+} \right)
\exp(- c_3 \delta_+^{\nu_+ + \eps_1+1})
\\
& \ls \exp(- c_4 \delta_+^{\nu_+ + \eps_1+1}).
\end{aligned}
\end{equation}
The term $\|\xi'' g\|_{L^2(\Real_+)}$ is estimated similarly
(and in fact more easily).

Putting everything together, we obtain (with some $c_5>0$)
\begin{equation}
\frac{\|\xi'' g\|_{L^2(\Real_+)} + \|\xi' g'\|_{L^2(\Real_+)}}{\|\xi g\|_{L^2(\Real_+)}} 
\ls
\exp(- c_5 \delta_+^{\nu+\eps_1+1}).
\end{equation}

If $V$ is bounded at $+\infty$, the appropriate rate in \eqref{kappa.est} follows immediately from \eqref{|g|.est} and the selected size of $\delta_\pm$ and $\Delta_\pm$, see \eqref{delta.def}. 
\end{proof}

\subsection{Remainder estimate}
\label{subsec:rem.unbdd}

\begin{Theorem}\label{thm:basic}
Let Assumption~\ref{asm:basic} hold and set $n:=N-1$.
Let $g$ be as in \eqref{g.def} with $\{\psi'_k\}_{k\in[[-1,n-2]]}$ determined by \eqref{ODE}, $\xi$ be as in \eqref{xi.def} with $\delta_\pm,$ $\Delta_\pm$ as in \eqref{delta.def} and $f$ be as in \eqref{Ansatz}. 
Then
\begin{equation}\label{HV.f.dec}
\frac{\|(H_V-\lambda) f\|}{\|f\|} = \kappa(\lambda) + \sigma^{(n)}(\lambda), 
\end{equation}
where $\kappa$ is as in \eqref{kappa.est} 
and $\sigma^{(n)} = \sigma_-^{(n)}  + \sigma_+^{(n)}$ with, as $\lambda \to + \infty$,
\begin{enumerate}[i)]
	\item if $V$ is unbounded at $\pm \infty$
\begin{equation}
\sigma_\pm^{(n)}(\lambda) =
\begin{cases}
\BigO(\lambda^{-\frac{n+1}2} \sup_{x \in \cJ_\pm}{|V(x)|\langle x \rangle^{(n+1)\nu_\pm}}  ), & \nu_\pm < 0,
\\[1mm]
\BigO \left(\delta_\pm^{(n+1) \nu_\pm} \lambda^\frac{1-n}{2} \right),
& \nu_+ \geq 0,
\end{cases}
\end{equation}
\item if $V$ is bounded at $\pm \infty$
\begin{equation}\label{HV.f.dec.bdd}
\sigma_\pm^{(n)}(\lambda)=
\begin{cases}
	\BigO \left( \lambda^{-\frac{n+1}{2}} \right), & \nu_\pm <0,
	\\[2mm]
	\BigO \left( \lambda^{-\frac{n+1}{2}\left(1-(1+\eps_2)\nu_\pm\right)} \right), & \nu_\pm \geq 0.
\end{cases}
\end{equation}

\end{enumerate}

\end{Theorem}
\begin{proof}
We employ the pseudomode construction for $n=N-1$. 
The estimate of the remainder $r_n$, see \eqref{rn.est}, 
and the assumption~\eqref{asm.eq:V'} 
together with~\eqref{lam.V.comp} and~\eqref{lambdapul}
yield that for $x>0$ and $V$ unbounded at $+\infty$ 
we have
\begin{equation}
|r_n(x )|
\ls
\begin{cases}
|V(x)|\langle x \rangle^{(n+1)\nu_+}\lambda^{-\frac{n+1}2}, & \nu_+ < 0, \\[1mm]
\delta_+^{(n+1) \nu_+} \lambda^\frac{1-n}{2} , & \nu_+ \geq 0,
\end{cases}
\end{equation}
and similarly for $x<0$. 
Here the case $\nu_+ \geq 0$ also employs
$\lambda \gs \langle \delta_+ \rangle^{4\nu_++2\eps_1+2}$,
which is a consequence of~\eqref{asm.eq:ReV} and~\eqref{delta.def}.
If $V$ is bounded at $\pm \infty$, the estimate of $r_n$ follows straightforwardly from \eqref{rn.est}, assumptions~\eqref{asm.eq:V'}, \eqref{asm.eq:ReV} 
and the choice of $\delta_\pm$ in \eqref{delta.def}.
\end{proof}

\subsection{Examples}
\label{subsec:Ex.1}

\begin{Example}[Polynomial-like potentials]
	\label{ex:pol.1}
Consider $V$ satisfying Assumption~\ref{asm:basic} with $\nu_-=\nu_+=-1$ 
and having the form 
\begin{equation}
  V := P_\beta + i Q_\gamma 
  \,,
\end{equation}
where $P_\beta$ and $Q_\gamma$ are real-valued functions satisfying 
\begin{equation}\label{PandQ}
\forall |x| \gs 1, \qquad 
|P_\beta(x)| \ls \langle x \rangle^\beta,
\qquad 
|Q_\gamma(x)| \approx \langle x \rangle^{\gamma},
\end{equation}
with some numbers $\beta\in\Real$ and $\gamma \geq 0$.
Typical examples of $P_\beta$ and $Q_\gamma$
are polynomials of degree $\beta$ and $\gamma$, respectively. 
Notice that a necessary condition to satisfy~\eqref{asm.eq:ImV} 
is $\gamma \geq 0$, while the sufficient one, 
which additionally guarantees~\eqref{asm.eq:ReV}, 
requires $\gamma > (\beta-2)/2$.
In particular for $\beta <2$ (\ie~$|\Re V(x)|$ 
grows slower than $x^2$) even a bounded $\Im V$ fits.

We define the quantity
\begin{equation}\label{omega.def}
  \omega := \max\{\beta,\gamma\} \geq 0
\end{equation}
and observe that $\omega=0$ if, and only if, $V$ is bounded.
If~$\omega$ is positive, then \eqref{delta.def} yields
\begin{equation}\label{delta.estimate}
\delta =\delta_- = \delta_+ \approx 
\lambda^{\frac1{2(\gamma+1)}+\epsilon},
\end{equation}
where $\epsilon>0$ can be made arbitrarily small by an appropriate choice of (small) $\eps_1>0$.
Hence the application of Theorem~\ref{thm:basic} yields  (with $n:=N-1$)
\begin{equation}\label{HVW.f.pol.pre}
\frac{\|(H_V -\lambda) f\|}{\|f\|} = 
\begin{cases}
\BigO \left(\lambda^{- \frac{n+1}2} \right), & \omega \leq n+1, 
\\[2mm]
\BigO \left(\lambda^{- \frac{n+1}{2} 
+ \frac{\omega-n-1}{2(\gamma+1)} + \epsilon(\omega-n-1))  } \right), 
& \omega > n+1, 
\end{cases}
\end{equation}
as $\lambda \to +\infty$.
Notice that the first case particularly involves bounded potentials (because $N \geq 1$)
and that the decay rate in the second case improves by diminishing~$\epsilon$.
It is also worth noticing that the restrictions on~$\beta$ and~$\gamma$ made above
imply the uniform bounds
\begin{equation}\label{uniform}
  \frac{\omega-n-1}{2(\gamma+1)} < 
  \begin{cases}
    1/2 & \mbox{if} \quad \gamma \geq \beta \,,
    \\
    1 & \mbox{if} \quad \gamma < \beta \,,
  \end{cases}
\end{equation}
which provides a rough estimate on the decay rate 
in the second case of~\eqref{HVW.f.pol.pre}.

Observe that the pseudomode with $n=1$ 
(\ie~we require $N\geq 2$) 
is sufficient to treat all polynomial-like potentials. 
The pseudomode with $n=0$ (\ie~$N\geq1$)
suffices for potentials growing not faster than linearly.
Notice also that for smooth potentials ($N=\infty$) the obtained rate is faster 
than any power of $\lambda^{-1}$.
\end{Example}

\begin{Example}[Exponential potentials]
	\label{ex:sinh}
Consider $V$ satisfying Assumption~\ref{asm:basic} 
with $\nu_-=\nu_+=0$ and $N\geq 3$; a simple smooth choice is 
$V(x) := \cosh x + i \sinh x$. 
Since $|V(x)| \ls e^{|x|}$, see \eqref{Gronw.est}, we have for sufficiently large $\lambda>0$ that 
\begin{equation}
\delta =\delta_- = \delta_+ \approx \log \lambda.
\end{equation}
Theorem~\ref{thm:basic} then gives
\begin{equation}
\frac{\|(H_V -\lambda) f\|}{\|f\|} = 
\BigO \left(\lambda^{\frac{2-N}2} \right), 
\end{equation}
thus exponential-type potentials can be treated using the pseudomodes with $n=2$.
\end{Example}

\begin{Example}[Bounded oscillating potentials]
	\label{ex:arctan}
Consider two smooth potentials
\begin{equation}
V_1(x) := i \arctan x, \qquad 
V_2(x) :=  2 i \arctan x +  i \sin \left(\langle x\rangle^{1+\mu} \right), 
\quad  0 < \mu <1.
\end{equation}
Clearly, $\nu_\pm=-2$ for $V_1$, however $\nu_\pm = \mu$ for $V_2$. Since both potentials are smooth, we can achieve an arbitrarily fast decay in \eqref{HV.f.dec.bdd} in both cases by taking $N$ large, nevertheless, substantially more terms in the pseudomode construction must be taken in the second case if $\mu$ is close to $1$.
\end{Example}

\subsection{Decaying potentials}
\label{subsec:dec}
Finally, we discuss a class of potentials that do not satisfy 
the basic assumption~\eqref{asm.eq:ImV},
but the method of the present section still enables one 
to construct the desired pseudomodes. Indeed,
the inequalities \eqref{|g|.est} suggest that the assumption \eqref{asm.eq:ImV} can be relaxed basically to $\Im V \notin L^1(\Real)$ if $\Im V$ 
has an appropriate sign for $x \gs 1$ and $x \ls 1$. Here we analyse the simplest examples, 
namely a class of smooth potentials with the asymptotic behaviour 
\begin{equation}
V(x) := i \, \frac{\sgn (x)}{\langle x \rangle^\gamma} \,, 
\qquad |x| \gs 1	, \quad 0 < \gamma <1.
\end{equation}
Since the essential spectrum of $H_V$ with this potential covers $[0, + \infty)$ and the numerical range of $H_V$ is a shrinking neighbourhood of this set, we will consider $\lambda = a + i b$ with $a \to +\infty$  and $b \to 0+$. 

The selection of a suitable $\delta_\pm$ for the cut-off 
is inspired by the estimate for $x \gs 1$
(the case $x \ls - 1$ and upper bounds are simpler)
\begin{equation}\label{Re.psi-1.dec}
\begin{aligned}
\int_0^x \Re (\lambda^\frac 12 \psi_{-1}'(t)) \, \der t
\gs
a^{-\frac12}
\int_0^x
(\langle t \rangle^{-\gamma} - b ) \; \der t
\gs
\frac{x^{1-\gamma}[1-(1-\gamma)bx^\gamma]
-C}{a^\frac12}
\end{aligned}
\end{equation}
with some $C \geq 0$. 
Thus, requiring that the first term in the expansion~\eqref{g.def}
leads to an integrable exponential, sought restrictions on $\delta_+$ read 
\begin{equation}\label{delta.dec}
a^\frac12 \delta_+^{\gamma-1} + b\delta_+^\gamma = o(1), \quad \lambda \to \infty;
\end{equation}
$\delta_-$ can be selected similarly 
and we can take $\Delta_\pm := \delta_\pm/4$. 
It can be also checked that the other terms in the expansion are negligible. Since $V$ is bounded, it is clear that the cut-off works and we indeed have a decay like in \eqref{cut.rates}. Regarding the remainders $r_n$, by taking sufficiently many terms in the expansion, we obtain a decay in \eqref{objective.intro} that is faster than any power of $1/a$. 

The set $\Omega$ where \eqref{objective.intro} holds can be obtained from \eqref{delta.dec}; in detail, we need 
\begin{equation}
b a^\frac{\gamma}{2(1-\gamma)} = o(1), \quad \lambda \to \infty.
\end{equation}
Observing that $V \in L^p(\Real)$ if $p\gamma >1$, we can further describe $\Omega$ by a condition essentially appearing in \cite[Thm.~5]{Laptev-Safronov_2009}:
\begin{equation}\label{Omega.dec}
b^{p-1} = o(a^{-\frac 12}), \quad \lambda \to \infty.
\end{equation}

\section{Lower regularity}
\label{sec:low.reg}

Our goal in this section is to treat potentials of lower regularity. The first possibility is a perturbative approach, \ie~we
search for conditions on a possibly singular perturbation $W$ guaranteeing that the pseudomodes constructed for a regular part~$V$, 
thus  ignoring~$W$, still exhibit a decay in \eqref{objective.intro}. The second option is to introduce a $\lambda$-dependent mollification $W^\eps$ of $W$ with $\eps=\eps(\lambda)$ and perform the construction for $V + W^\eps$; naturally the crucial point is to determine suitable dependence of the mollification on $\lambda$.

In both approaches, eventually, we need more precise information on the $L^p$-norms of pseudomodes. We make here additional assumptions on the growth of $V$; in fact we analyse in detail potentials with a polynomial growth, nonetheless, other cases may be treated similarly.

\subsection{Weighted $L^p$-norms of pseudomodes}

\begin{Lemma}\label{lem:Lp}
Let Assumption~\ref{asm:basic} hold, let $f$ be as in \eqref{Ansatz} 
with  $0 \leq n \leq N$. Then for all sufficiently large $\lambda >0$ 
the following holds.
\begin{enumerate}[i)]
\item If there is $\gamma \geq 0$ such that
\begin{equation}\label{ImV.lb}
\forall x \gs 1, \quad |\Im V(x)| \ls |x|^\gamma, 
\qquad \text{or} \qquad  
\forall x \ls -1, \quad |\Im V(x)| \ls |x|^\gamma, 
\end{equation}
then
\begin{equation}\label{f.p.lb}
\|f\|_p \gs \lambda^\frac{1}{2p(\gamma+1)}, \qquad 2 \leq p \leq \infty.
\end{equation}
\item If there are $\gamma_\pm \geq 0$ such that
\begin{equation}\label{ImV.ub}
|\Im V(x)| \gs
\begin{cases}
|x|^{\gamma_+}, & x \gs 1,
\\ 
|x|^{\gamma_-}, & x \ls -1,
\end{cases} 
\end{equation}
then
\begin{equation}\label{f.gamma.p.ub}
\|\langle x \rangle^\alpha f(x)\|_{L^p(\Real_\pm)} \ls \lambda^\frac{1+p\alpha}{2p(\gamma_\pm+1)}, \qquad 2 \leq p \leq \infty, \quad \alpha \geq 0 .
\end{equation}
\end{enumerate}
\end{Lemma}
\begin{proof}
i) Suppose that the first inequality in \eqref{ImV.lb} holds. From \eqref{|g|.est} we have (with some $C \geq 0$, $c>0$)
\begin{equation}
\|f\|_p^p 
\gs 
\int_C^{\delta_+-\Delta_+}  e^{-pc \lambda^{-\frac12}  x^{\gamma+1}}  \; \der x
= 
\lambda^\frac1{2(\gamma+1)} \int_{C\lambda^{-\frac1{2(\gamma+1)}}}^{(\delta_+-\Delta_+) \lambda^{-\frac1{2(\gamma+1)}}} e^{-pc y^{\gamma+1}}  \; \der y.
\end{equation}
Thus it remains to verify that $(\delta_+-\Delta_+) \lambda^{-\frac1{2(\gamma+1)}} \gs 1$. The latter follows from \eqref{delta.def}. The case of a bounded $V$ is simple and in the unbounded case (necessarily with $\nu_+ \geq -1$, see \eqref{Gronw.est}) we get from \eqref{ImV.lb} and \eqref{delta.def} that
\begin{equation}
\frac{\delta_+^{2(\gamma+1)}}{\lambda }
\approx
\frac{\delta_+^{2(\gamma+1) + 4 \nu_+ + 2 \eps_1 +2}}{|\Im V(\delta_+)|^2}
\gs \delta_+^{4\nu_+ + 4 + 2 \eps_1} \gs 1.
\end{equation}
This proves~\eqref{f.p.lb} for $p \in [2,\infty)$
under the first of the assumptions in~\eqref{ImV.lb},
the second alternative is treated similarly.
The case $p=\infty$ is even simpler to show.

ii) For $x \geq 1$ we have $\langle x \rangle \approx x$, thus \eqref{|g|.est} and \eqref{ImV.ub} yield (with some $C \geq 1$, $c>0$)
\begin{equation}
\left( \int_0^C +  \int_C^{\delta_+} \right) \langle x \rangle^{p\alpha} e^{-p c \lambda^{-\frac 12} x^{\gamma_+ + 1}} \; \der x
\ls
1+
\lambda^\frac{1+p \alpha}{2(\gamma_++1)}
\int_0^{\infty}  y^{p\alpha} e^{-p c y^{\gamma_+ + 1}} \; \der y.
\end{equation}
The case $p=\infty$ can be checked by calculating the maximum of $|f|$ and the second case for $x \ls -1$ is analogous.
\end{proof}

The immediate consequence is a possibility to employ pseudomodes constructed for $V$ even for $V+W$, where $W$ is an $L^r$-perturbation. 

\begin{Theorem}\label{thm:Lr}
Let Assumption~\ref{asm:basic} hold and set $n:=N-1$. 
Let $\Im V$ satisfy \eqref{ImV.lb} and \eqref{ImV.ub} and let $W \in L^{r_-}(\Real_-) + L^{r_+}(\Real_+)$ with some $2 \leq r_\pm < \infty$. Then
\begin{equation}\label{HVW.f.ign}
\frac{\|(H_{V+W}-\lambda) f\|}{\|f\|} 
= \kappa(\lambda) + \sigma^{(n)}(\lambda) + \rho(\lambda),
\end{equation}
where $f$, $\kappa$ and $\sigma^{(n)}$ are as in Theorem~\ref{thm:basic} 
and $\rho = \rho_- + \rho_+$ with
\begin{equation}
\rho_\pm(\lambda) = \BigO\left(\lambda^\frac{\gamma-\gamma_\pm - \frac 2{r_\pm} (\gamma+1)}{4(\gamma_\pm+1)(\gamma+1)}\right), \qquad \lambda \to + \infty,
\end{equation}
where~$\gamma$ and~$\gamma_\pm$ are as in Lemma~\ref{lem:Lp}.
\end{Theorem}
\begin{proof}
The estimate follows from \eqref{f.p.lb}, \eqref{f.gamma.p.ub} with $\alpha=0$ and H\"older inequality. In detail, with $2/r_\pm+ 2/s_\pm =1$, we have
\begin{equation}
\frac{\|Wf\|_{L^2(\Real_\pm)}}{\|f\|} 
\leq 
\frac{\|W\|_{L^{r_\pm}(\Real_\pm)}\|f\|_{L^{s_\pm}(\Real_\pm)}}{\|f\|} 
\ls 
\lambda^\frac{2(\gamma+1) -s_\pm(\gamma_\pm+1)}{4s_\pm(\gamma+1)(\gamma_\pm+1}
\end{equation}
and the claim follows when $s_\pm$ is  expressed in terms of $r_\pm$.
\end{proof}
The weighted $L^p$-estimates of $f$ can be used also to employ the pseudomode with $n=N$, instead of $n=N-1$ in Theorem~\ref{thm:basic}, and thereby lower assumptions on the regularity of $V$. 
\begin{Theorem}\label{thm:N+1}
Let Assumption~\ref{asm:basic} hold and set $n:=N$. 
Let $\Im V$ satisfy \eqref{ImV.lb} and \eqref{ImV.ub} 
and let $V^{(N+1)} \in L^2(\Real) + L^\infty_{-\alpha_-}(\Real_-) + L^\infty_{-\alpha_+}(\Real_+)$ with some $\alpha_\pm \geq 0$. 
Then
\begin{equation}\label{HVW.f.N+1}
\frac{\|(H_V-\lambda) f\|}{\|f\|} 
= \kappa(\lambda) + \sigma^{(n)}(\lambda) + \tau(\lambda),
\end{equation}
where $f$ is the pseudomode with $n=N$, $\kappa$ 
and $\sigma^{(n)}$ are as in Theorem~\ref{thm:basic} and $\tau= \tau_- + \tau_+$ with
\begin{equation}
\tau_\pm(\lambda) = 
\BigO
\left(
\lambda^{-\frac{N+1}{2} - \frac{1}{4(\gamma+1)}}
+
\lambda^{-\frac{N+1}{2} + \frac{\gamma-\gamma_\pm 
+ 2 \alpha_\pm  (\gamma+1) }{4(\gamma_\pm+1)(\gamma+1)}}
\right), \qquad \lambda \to + \infty,
\end{equation}
where~$\gamma$ and~$\gamma_\pm$ are as in Lemma~\ref{lem:Lp}.
\end{Theorem}
\begin{proof}
If $f$ is taken as the pseudomode with $n=N$, the terms $\kappa$ 
and $\sigma^{(n)}$ in \eqref{HVW.f.N+1} are estimated in the same way as in Theorem~\ref{thm:basic}. The difference arises in the first term of $r_n$, 
see \eqref{rn.est}, since it contains $V^{(N+1)}$, more precisely, we need to estimate
\begin{equation}
\lambda^{-\frac{N+1}{2}} \|V^{(N+1)}f\|. 
\end{equation}
The claim follows straightforwardly from the assumption on $V^{(N+1)}$, H\"older inequality, \eqref{f.p.lb} and \eqref{f.gamma.p.ub}.
\end{proof}

\subsection{Examples}
\begin{Example}[Singularly perturbed polynomial-like potentials]
	\label{ex:pol.2}
Let~$V$ be as in Example \ref{ex:pol.1} and $W \in L^{r_-}(\Real_-) + L^{r_+}(\Real_+)$ with $2 \leq r_\pm <\infty$. If Assumption~\ref{asm:basic} holds with $N \geq 2$, 
Theorem~\ref{thm:Lr} and the already obtained rates $\sigma^{(n)}$, 
see Example~\ref{ex:pol.1} and in particular~\eqref{uniform}, yield
\begin{align}\label{HVW.f.pol}
\frac{\|(H_{V+W}-\lambda) f\|}{\|f\|} 
&= 
\BigO\left(
\lambda^{-\frac{ 1 }{2 r_\pm(\gamma+1)}} 
\right)
+
\begin{cases}
\BigO \left(\lambda^{- \frac N2} \right), & \omega \leq N, 
\\[1mm]
\BigO \left(\lambda^{- \frac{N}{2} + \frac{\omega-N}{2(\gamma+1) + \epsilon(\omega-N)  }} \right)
& \omega > N,
\end{cases}
\nonumber \\
&= \BigO\left(
\lambda^{-\frac{ 1 }{2 r_\pm(\gamma+1)}} 
\right), 
\end{align}
as $\lambda \to + \infty$.
Here the second equality follows by the restrictions
made on~$\beta$ and~$\gamma$ in Example \ref{ex:pol.1}
(\cf~particularly~\eqref{uniform}).
In other words, adding the singularity~$W$ 
deteriorates the decay rate~\eqref{HVW.f.pol.pre}
(at least when the result of Theorem~\ref{thm:Lr} is used).
\end{Example}

\begin{Example}[Imaginary step-like potential]
\label{ex:sgn.1}
Now we would like to treat
the discontinuous example from~\eqref{step-like}.
First, to apply Theorem~\ref{thm:Lr},
we specify a suitable splitting (to have a sufficiently regular~$V$)
\begin{equation}\label{splitting}
V(x) := i (1-\eta(x))\sgn (x), \qquad 
W(x) := i \eta(x) \sgn (x), 
\end{equation}
with some $\eta \in C_0^\infty((-1,1))$ and $\eta = 1$ on a neighbourhood of $0$. 
Then Theorem~\ref{thm:Lr} 
(with $N \geq 1$, $r_\pm:=2$ and $\gamma_\pm:=0=:\gamma$)
yields
\begin{equation}\label{HVW.f.sgn.1}
\frac{\|(H_{i \sgn}-\lambda) f\|}{\|f\|} = 
\BigO \left( \lambda^{-\frac 14} \right), \qquad \lambda \to + \infty.
\end{equation}
\end{Example}
\begin{Example}[Polynomial growth with a local singularity]
As an application of Theorem~\ref{thm:N+1}, let us consider
the following class of potentials
$$
  V(x) := i \, \sgn (x) \, |x|^\gamma
  \left(2+\sin |x|^{-\mu}\right)
  , \qquad 
  \mu \in (0,1) 
  \,, \ 
  \gamma \in \Nat 
  \,.
$$
If $\gamma \geq 2$ and $N:=\gamma-1$,
it is easy to verify that~$V$ also satisfies 
the other items of Assumption~\ref{asm:basic} (with $\nu_\pm := -1$),
namely the basic regularity requirement $V \in W^{N,\infty}(\Real)$.
Since the derivative $V^{(\gamma)}$ has a singularity at zero, however,
the best decay rate we can obtain by directly
applying Theorem~\ref{thm:basic} is 
\begin{equation} 
\frac{\|(H_{V}-\lambda) f\|}{\|f\|} = 
\BigO\left(
\lambda^{-\frac{\gamma-1}{2} + \frac 1{2(\gamma+1)} + \epsilon}  
\right), \qquad \lambda \to + \infty,
\end{equation}
where $\epsilon>0$ can be made arbitrarily small
(\cf~\eqref{delta.estimate}).
On the other hand, observing that 
$V^{(\gamma)} \in \sii(\Real) + L^\infty(\Real)$ 
and applying Theorem~\ref{thm:N+1} 
(with $\alpha_\pm:=0$ and $\gamma_\pm:=\gamma$), 
where the pseudomode with one more term in the expansion is employed, 
we obtain a better result, namely
\begin{equation} 
\frac{\|(H_{V}-\lambda) f\|}{\|f\|} = 
\BigO\left(
\lambda^{-\frac{\gamma}{2}} 
\right), \qquad \lambda \to + \infty.  
\end{equation}
\end{Example} 

\subsection{Mollification strategy}
Now we turn to the alternative approach to deal
with irregular potentials.

Let $ w \in C_0^\infty(\Real)$ with $0 \leq w \leq 1$, 
$\supp w = [-1,1]$ and $\|w\|_1=1$ and define
\begin{equation}\label{w.eps.def}
w_\eps(x):= \frac{1}{\eps} \, w \left(\frac x \eps\right), \qquad x \in \Real, \quad \eps >0.
\end{equation}
For $\phi \in L^p_{\rm loc}(\Real)$, we introduce the $L^p$ modulus of continuity on an interval  $\cJ \subset\Real$ by
\begin{equation}\label{omega.p.def}
\omega_p(\eps;\phi,\cJ):= \sup_{0 < |t| < \eps} 
\|\phi(\cdot + t) - \phi\|_{L^p(\cJ)}, \quad 1 \leq p < \infty.
\end{equation}
Finally, we introduce an $\eps$-neighbourhood of $\cJ$,
$\cJ_\eps := \{x \in \Real \, : \, \dist(x,\cJ) < \eps \}$.

The main idea in what follows is the mollification of a singular part of the potential. 
For $\phi \in L^p_{\rm loc}(\Real)$ and $w_\eps$ as in \eqref{w.eps.def}, we denote
\begin{equation}\label{phi.eps.def}
\phi^\eps : = w_\eps * \phi.
\end{equation}

To be able to estimate newly constructed pseudomodes, we need several basic properties of mollifications and their relation to the $L^p$ modulus of continuity summarised in the following lemma; the proof relies on Minkowski's integral inequality
and properties of the convolution and of~$w$. 
\begin{Lemma}\label{lem:mol}
Let $\phi \in L^p_{\rm loc}(\Real)$ with $1 \leq p < \infty$, $\phi^\eps$ be as in \eqref{phi.eps.def}, $\cJ$ be an interval and $\cJ_\eps$ its $\eps$-neighbourhood. 
Then for every $1 \leq p < \infty$, $j \in \Nat$ and $\eps>0$ we have
\begin{align}
\|\phi^\eps \|_{L^p(\cJ)} &\leq \|\phi \|_{L^p(\cJ_\eps)},
&
\|\phi^\eps \|_{L^\infty(\cJ)} &\leq \eps^{-\frac1p} \, \|\phi \|_{L^p(\cJ_\eps)}, \label{St.inf}
\\ 
\|\phi-\phi^\eps\|_{L^p(\cJ)} &\leq \omega_p\left(\eps; \phi, \cJ \right),
&
\|(\phi^\eps)^{(j)}\|_{L^p(\cJ)} &\leq \eps^{-j} \, 
\omega_p\left(\eps; \phi, \cJ \right)
\|w^{(j)}\|_{L^1}. 
\label{St.approx}
\end{align}
\end{Lemma}

We proceed with the construction of pseudomodes for a potential $V+W$ where~$W$ is possibly discontinuous and singular. In fact the pseudomodes are constructed for $V+W^\eps$ with a certain $\lambda$-dependent mollification. Thus besides usual remainders \eqref{rem.phi} we need to estimate also $\|(W-W^\eps)f\|$. 

While the construction can be in principle performed with an arbitrary number of terms, we restrict ourselves to the case $n\in [[0,1]]$ since the assumptions on the singular part $W$ would become more complicated and implicit for $n>1$. In spite of this restriction, we can still treat potentials with $\nu_\pm <0$, \ie~even with some super-polynomial growth or oscillations. 
More precisely, new pseudomodes are constructed under the following assumptions. 
\begin{Assumption}\label{asm:W}
Let $V$ satisfy Assumption~\ref{asm:basic} with $N \in [[1,2]]$ and $\nu_\pm <0$ and suppose that $W=W_1 + W_2$ satisfy
\begin{enumerate}[a)]
	\item $|\Im W_1| \leq (1-\eps) |\Im V|$ with some $0<\eps <1$ and with $\eps_1>0$ from Assumption~\ref{asm:basic}	
\begin{equation}\label{asm.eq:ReW_1}
\forall x \in \Real_\pm, \quad |\Re W_1(x)| \ls |\Im V(x)|^2 \langle x \rangle^{-4(\nu_\pm + \eps_1)-2}.
\end{equation}
	\item $W_2 \in L^2(\Real)$ and $\supp W_2$ is compact.
\end{enumerate}
\end{Assumption}

The mollification~\eqref{phi.eps.def}
is done separately for three parts of~$W$, namely
\begin{equation}\label{W.mol}
\tilde W:=(\chi_-W_1)^{\eps_-} + (\chi_+W_1)^{\eps_+}+ W_2^{\eps_0}
\end{equation}
with $\chi_\pm$ being the characteristic function of $\Real_\pm$ and
\begin{equation}\label{eps.iota.def}
\eps_\iota := \lambda^{-\alpha_\iota}, 
\quad \alpha_\iota \in (0,1), 
\quad \iota \in \{-,+,0\}.
\end{equation}
The expansion (see \eqref{g.def})
\begin{equation}\label{moly}
\tilde g
  := \exp\left(- \sum_{k=-1}^{n-1}\lambda^{-\frac k2} \psi_k(x)\right), 
\quad  n \leq 1,
\end{equation}
is determined by functions $\psi_{k}'$ satisfying~\eqref{psi'.list}
with $V$ replaced by 
\begin{equation}\label{V.mol}
\tilde V := V + \tilde W. 
\end{equation}
On the other hand, we keep the size of the cut-off the same as for~$V$ only, 
\ie~the new pseudomodes read
\begin{equation}\label{f.mol}
\tilde f := \xi \tilde g, 
\end{equation}
where $\xi$ is as in \eqref{xi.def} with $\delta_\pm$, $\Delta_\pm$ as in \eqref{delta.def} with $V$.

\begin{Lemma}
Let Assumption~\ref{asm:W} hold and 
$\tilde g$ be as in~\eqref{moly} with~\eqref{V.mol}.
Then 
\begin{equation}\label{cut.rates.mol}
\kappa(\lambda) :=
\frac{\|\xi'' \tilde g\| + \|\xi' \tilde g'\|}{\|\tilde g\|} 
= o(1), \quad \lambda \to + \infty,
\end{equation}
with $\kappa$ as in \eqref{kappa.est} (with possibly a smaller positive constant $c>0$).
\end{Lemma}
\begin{proof}
We start with showing 	
	
\begin{equation}\label{f.lb.mol}
\|\tilde f\|^2 \gs \int_\cJ \exp \left(-c_3 \lambda^{-\frac 12} \int_0^{|x|} |\Im V(t)| \, \der t  \right) \, \der x
\end{equation}
with some $c_3>0$, where $\tilde f$ is defined in~\eqref{f.mol}. 
We give details on estimates on $\Real_+$, the other case is analogous. 
First notice that $\tilde{W}$ is locally bounded, see Lemma~\ref{lem:mol}.
Moreover, since $\eps_\pm = o(\Delta_\pm)$, 
we obtain from \eqref{asm.eq:V'}, \eqref{V.x.Delta} and assumptions on $W$ that
\begin{equation}\label{ImW1.V}
\begin{aligned}
|\Im(\chi_\pm W_1)^{\eps_\pm}(x)| 
&\leq 
\int_{\Real} w_{\eps_\pm}(y) |\Im W_1(x-y)| \; \der y 
\leq 
\sup_{|y|<\eps_\pm} |\Im W_1(x-y)| 
\\
&\leq 
(1-\eps)\sup_{|y|<\eps_\pm} |\Im V(x-y)|
\\
& \leq (1-\eps) \left(|\Im V(x)| 
+ \sup_{|y|<\eps_\pm} \left| \int_x^{x-y}|\Im V'(t)| \, \der t \right| \right)
\\
&\leq (1-\eps)|\Im V(x)| (1+ \BigO(\eps_\pm))
\end{aligned}
\end{equation}
and similarly, using \eqref{asm.eq:ReW_1} and \eqref{V.x.Delta},
\begin{equation}\label{ReW1.V}
\forall x \in \Real_\pm, \quad |\Re (\chi_\pm W_1)^{\eps_\pm}(x)| 
\ls 
|\Im V(x)|^2 \langle x \rangle^{-4(\nu_\pm + \eps_1)-2}.
\end{equation}
For $W_2$, Lemma~\ref{lem:mol} yields immediately
\begin{equation}\label{W2.est}
|W_2^{\eps_0}(x)| \leq \eps_0^{-\frac 12} \|W_2\| = o\left(\lambda^\frac12 \right), \quad \lambda \to +\infty.
\end{equation}

The estimates above imply that $\tilde{W}$ 
can be absorbed by~$V$ or~$\lambda$ in all relevant estimates in Lemmata~\ref{lem:exp.bound}, \ref{lem:g.est} and Proposition~\ref{prop:cut}; in particular notice that~$W_2$ affects the estimates only on a compact set due to the assumed boundedness of $\supp W_2$, and that the size of $\Re \tilde W$ is the largest possible complying with \eqref{asm.eq:ReV} and \eqref{delta.def}. 
Straightforward estimates of \eqref{g.formula} with $n \in [[0,1]]$ 
and with~$V$ replaced by~$\tilde V$ lead to (with some $c_1,c_2 >0$)
\begin{equation}\label{|g|.est.mol}
e^{-c_1 \lambda^{-\frac12}  \int_0^{|x|} |\Im V(t)| \, \der t} 
\ls | \tilde g(x)| \ls  
e^{-c_2 \lambda^{-\frac12}  \int_0^{|x|} |\Im V(t)| \, \der t}
\end{equation}
for $n \in [[0,1]]$, all sufficiently large $\lambda$ and all $x \in \cJ$; 
here \eqref{ImW1.V}, \eqref{ReW1.V}, \eqref{W2.est} 
and the boundedness of $\supp W_2$ were used.
Hence \eqref{f.lb.mol} follows. 

To verify \eqref{cut.rates.mol}, we need in addition that 
\begin{equation}\label{W1'.V}
\forall x \in \cJ_\pm, \quad
|((\chi_+W_1)^{\eps_\pm})'(x)| \ls \frac{|\Im V(x)|+ |\Im V(x)|^2 \langle x \rangle^{-4(\nu_\pm + \eps_1)-2}}{\eps_\pm};
\end{equation}
the proof si similar to \eqref{ImW1.V} and \eqref{ReW1.V}.
Hence, using \eqref{psi'.list}, \eqref{W1'.V} and \eqref{eps.iota.def}, we obtain
\begin{equation}
\forall x \in \cJ_\pm, \quad \lambda^\frac 12 |\psi_{-1}'(x)| + |\psi_0'(x)| \ls \lambda^\frac 12 + |V(x)| + |\Im V(x)|^2 \langle x \rangle^{-4(\nu_\pm + \eps_1)-2}.
\end{equation}
The rest of the proof is a simple modification of the one of Proposition~\ref{prop:cut}.
\end{proof}

Now we are in a position to state the main result of this section.
\begin{Theorem}\label{thm:mol}	Let Assumption~\ref{asm:W} hold 
and~$\tilde f$ be as in \eqref{f.mol} with $n \in[[0,1]]$. 
Then
\begin{equation}\label{HVW.f.mol}
\frac{\|(H_{V+W}-\lambda) \tilde f\|}{\|\tilde f\|} = \kappa(\lambda) + \sigma^{(n)}(\lambda) + \frac{\zeta^{(n)}(\lambda)}{\|\tilde f\|},
\end{equation}
where $\kappa$ and $\sigma^{(n)}$ are as in Theorem~\ref{thm:basic} and $\zeta^{(n)}= \zeta_-^{(n)} + \zeta_+^{(n)} + \zeta_0^{(n)}$ with, 
as $\lambda \to + \infty$,  
\begin{equation}
\begin{aligned}
\zeta_\pm^{(0)}(\lambda) 
& = \BigO \Big( \omega_2(\eps_\pm;W_1,\cJ_\pm) \left(1+  \eps_\pm^{-1} \lambda^{- \frac 12}  \right)
 \Big),
\\
\zeta_0^{(0)}(\lambda)  
& = \BigO \Big( \omega_2(\eps_0;W_2,\Real) \left(1+  \eps_0^{-1} \lambda^{- \frac 12}  \right) 
 \Big),
\\
\zeta_\pm^{(1)}(\lambda)  
& = \BigO \Big( \omega_2(\eps_\pm;W_1,\cJ_\pm) \left(1+  \eps_\pm^{-2} \lambda^{-1} \right)
+ \omega_4(\eps_\pm;W_1,\cJ_\pm) \eps_\pm^{-2} \lambda^{-2}
\Big),
\\
\zeta_0^{(1)}(\lambda)  
& = \BigO \Big( \omega_2(\eps_0;W_2,\Real) \left(1+ \eps_0^{-2} \lambda^{-  1} \right) 
+
\omega_4(\eps_0;W_2,\Real) \eps_0^{-2} \lambda^{-2} 
\Big),
\end{aligned}
\end{equation}
where $\eps_\iota$ are as in \eqref{eps.iota.def}. 
\end{Theorem}
\begin{proof}
Inserting the pseudomode $\tilde f$, we obtain
\begin{equation}\label{H.f.mol}
\|(H_{V+W}-\lambda) \tilde f\| 
\leq \|(H_{\tilde V}-\lambda) \tilde f\| + \|(\tilde W - W)\tilde f\|.
\end{equation}
We need to estimate remainders \eqref{rem.phi} with $\tilde V$ and the second term in \eqref{H.f.mol}. The claim follows straightforwardly from \eqref{rn.est.list} and the properties of the mollification, see Lemma~\ref{lem:mol}.
\end{proof}

\subsection{Examples}
\label{subsec:Ex.mol}

First we prove a lemma on the $L^p$ modulus 
of continuity of piece-wise $C^1$ potentials with a controlled growth.

\begin{Lemma}\label{lem:W1.om.p}
Let $W$ be a piece-wise $C^1$ function, more precisely $W \in C^1(\Real \setminus \mathcal M)$ with $\mathcal M:= \{a_k\}_{k \in \Int}$ such that for all  $k \in \Int,$ $a_{k+1} - a_k \gs 1$ and for all $k \in \Int$ the one-sided limits $\lim_{x \to {a_k}_\pm }W(x)$ exist and are finite. Moreover, let $W$ satisfy 
	\begin{equation}
	\exists \beta_\pm \in \Real, \quad \forall x \in \Real_\pm, \quad |W(x)| \ls \langle x \rangle^{\beta_\pm},
	\end{equation}
	and
	\begin{equation}
	\exists \gamma_\pm \in \Real, \quad \forall x \in \Real_\pm \setminus \mathcal M, \quad |W'(x)| \ls \langle x \rangle^{\gamma_\pm}.
	\end{equation}
	Then, for all~$\eps$ small and~$\delta_\pm$ large,
	\begin{equation}\label{om.p.W}
	\omega_p(\eps; W, \cJ_\pm) 
	\ls
	\eps \delta_\pm^{\gamma_\pm + \frac 1p} + \eps^\frac 1p \delta_\pm^{\beta_\pm + \frac 1p}, \quad 2 \leq p < \infty.
	\end{equation}
If in addition $\supp W$ is bounded, then
\begin{equation}\label{om.p.bdd.supp}
\omega_p(\eps; W, \Real) 
\ls
\eps^\frac 1p, \quad 2 \leq p < \infty.
\end{equation}

\end{Lemma}
\begin{proof}
We analyse only the case with $\cJ_+$,
the other situation is analogous. 
We can assume that $a_0=0$ and $a_{L+1}=\delta_+$ 
with some $ L \in \Nat$. 
Splitting the intervals $(a_k,a_{k+1})$ to $\eps$-neighbourhoods 
of the discontinuities and the rest and employing the assumptions on~$W$ and~$W'$, 
we have, for every $|t| < \eps$,
\begin{eqnarray*}
\lefteqn{
\int_{\cJ_+}|W(x+t)-W(x)|^p \, \der x
= \sum_{k=0}^L \int_{a_k}^{a_{k+1}} |W(x+t)-W(x)|^p \, \der x
}
\\
&& 
= \sum_{k=0}^L \int_{a_k}^{a_{k+1}-\eps} 
\left| \int_{x}^{x+t} W'(\xi) \, \der\xi \right|^p \, \der x
+ \sum_{k=0}^L \int_{a_{k+1}-\eps}^{a_{k+1}} |W(x+t)-W(x)|^p \, \der x
\\ 
&& 
\ls
\sum_{k=0}^L \int_{a_k}^{a_{k+1}-\eps} \!\!\! \der x \,
\Big(\esssup_{(a_0-\eps,a_{L+1})} |W'|\Big)^p \eps^p 
+\sum_{k=0}^L \int_{a_{k+1}-\eps}^{a_{k+1}} \!\!\! \der x \, 
\sup_{(a_{k+1}-\eps,a_{k+1}+\eps)} |W|^p 
\\ 
&& 
\ls \delta_+^{1+p\gamma_+} \, \eps^p  
+ \eps  \sum_{k=0}^L a_{k+1}^{p\beta_+}.
\end{eqnarray*}
Consequently, 
\eqref{om.p.W}~follows since $a_{k+1}-a_k \gs 1$ 
and the last sum can be estimated
by an integral (details are omitted).

If $\supp W$ is bounded, then the estimates are performed on a bounded interval independent of $\delta_\pm$ and so \eqref{om.p.bdd.supp} follows as well.
\end{proof}

\begin{Example}[Imaginary step-like potential continued] 
\label{ex:sgn.2}
Following Example~\ref{ex:sgn.1},
we keep the splitting of the imaginary sign potential~$i\sgn$
to the sum of the smooth potential~$V$ and the discontinuous~$W$
with a compact support, see~\eqref{splitting}.
The latter obeys Assumption~\ref{asm:W} with $W_1:=0$.
Applying Theorem~\ref{thm:mol} 
(with $n:=1$ and $\alpha_0:=1/2$ in~\eqref{eps.iota.def})
with help of Lemma~\ref{lem:Lp} (with $\gamma:=0$ and $p:=2$) to estimate $\|\tilde f\|$ 
and Lemma~\ref{lem:W1.om.p} to estimate 
the moduli of continuity in $\zeta^{(1)}(\lambda)$,
we arrive at
	\begin{equation*}
	\frac{\|(H_{i \sgn}-\lambda)\tilde f\|}{\|\tilde f\|} 
	= \BigO \left( \lambda^{-\frac12} \right), \qquad \lambda \to +\infty
	\,.  
	\end{equation*}
	This is an improvement with respect 
	to the rate $\lambda^{-\frac{1}{4}}$ 
	provided by Theorem~\ref{thm:Lr}, see Example~\ref{ex:sgn.1}.
	Nevertheless, 
	even this better rate is not optimal,
	as it is known from~\cite{HK} that there exists 
	a pseudomode with the decay rate $O(\lambda^{-1})$ 
	and that it is actually the best possible.
\end{Example}
\begin{Example}[Infinite steps]
Let us consider the step-like (odd) potential 
\begin{equation} 
 U(x) := i \, \lfloor |x| \rfloor^\gamma \, \sgn(x)\,,
  \qquad
  \gamma >0  
  \,,
\end{equation}
where~$\lfloor \cdot \rfloor$ denotes the floor function.
Hence~$U$ represents a piece-wise approximation of 
$x \mapsto i \, |x|^\gamma \sgn(x)$ 
(\cf~Example~\ref{ex:pol.1} with $P_\beta:=0$).
The basic hypothesis~\eqref{asm.eq:ImV} is clearly satisfied,
so it is expected that~$H_U$ admits pseudomodes.
However, Theorem~\ref{thm:basic} cannot be used because
of the lack of regularity required by Assumption~\ref{asm:basic}.

We show how Theorem~\ref{thm:mol} can be used instead. To this end, 
we split $U$ as
\begin{equation} 
  U = V + W \,,
  \qquad  
  W = W_1 + W_2
  \,,
\end{equation}
where
\begin{equation}
\begin{aligned}
  V(x) &:= i \, (1-\eta(x)) \, |x|^\gamma \sgn(x) 
  \,, \\
  W_1(x) &:= i \, (1-\eta(x)) 
  \left( \lfloor |x| \rfloor^\gamma - |x|^\gamma \right) 
  \sgn(x)
  \,, \\
  W_2(x) &:= i \eta(x) \, \lfloor |x| \rfloor^\gamma \, \sgn(x)
  \,,
\end{aligned}
\end{equation}
and $\eta \in C_0^\infty(\Real)$ is such that $0 \leq |\eta| \leq 1$
and $\eta=1$ on the interval $[-\gamma-1,\gamma+1]$.
Using the mean value theorem and properties of the floor function, 
we have
\begin{equation} 
  |W_1(x)| 
  \leq (1-\eta(x)) \,
  \gamma \, |x|^{\gamma-1} \, \big| \lfloor |x| \rfloor - |x| \big|
  \leq (1-\eta(x)) \,
  \gamma \, |x|^{\gamma-1} 
\end{equation}
for every $x\in\Real$.
Since~$W_1(x)$ equals zero if $|x| \leq \gamma+1$,
we see that Assumption~\ref{asm:W} holds with $\eps:=1/(\gamma+1)$.

Now we are in a position to apply Theorem~\ref{thm:mol} with $n:=1$. 
For $\sigma^{(1)}(\lambda)$, we always have a decay, 
see Example~\ref{ex:pol.1}.
Lemma~\ref{lem:Lp} with $p:=2$ yields 
\begin{equation}\label{f.estimate} 
\|\tilde f\| \gs \lambda^{\frac{1}{4(\gamma+1)}}
\end{equation}
for all sufficiently large~$\lambda$ 
and Lemma~\ref{lem:W1.om.p} immediately implies (we take $\alpha_0:=1/2$)
\begin{equation}\label{zeta0.estimate}  
\zeta_0^{(1)}(\lambda) = \BigO \left( 
\lambda^{-\frac14} 
\right)
, \qquad
\lambda \to +\infty \,.
\end{equation}
Again from Lemma~\ref{lem:W1.om.p} (with $\beta_\pm:=\gamma-1$ 
and $\gamma_\pm$ arbitrarily large negative), we obtain for $W_1$ 
(with $\alpha_\pm:=\alpha \in (0,1)$)
\begin{equation}
\lambda^{- \frac{1}{4(\gamma+1)}}\omega_p(\lambda^{-\alpha}; W_1, \cJ_\pm) 
= 
\BigO \left( \lambda^{-\frac \alpha p + \frac{1}{2p(\gamma+1)} + \frac{2\gamma-3}{4(\gamma+1)} + \epsilon}  \right), \qquad \lambda \to +\infty,
\end{equation}
where $\epsilon=\epsilon(\gamma,p)>0$ can be made arbitrarily small.
Calculating the individual terms in $\zeta_{\pm}^{(1)}$, 
we obtain the following conditions on $\alpha$ to have a decay in \eqref{HVW.f.mol}:
\begin{equation}
\frac{\gamma-1}{\gamma+1}
< \alpha < 
\frac13 \frac{\gamma+3}{\gamma+1} 
\,.
\end{equation}
These can be satisfied only if $\gamma<3$ and the corresponding decay rate in \eqref{HVW.f.mol} can be calculated in a straightforward way.
\end{Example}

\section{Pseudomodes for general curves}
\label{sec:pseudo.gen}
In this section, we focus on potentials~$V$ with unbounded $\Im V$
and investigate pseudomodes for other curves in the complex plane
than lines parallel to the real axis. The construction is basically the same as in Section~\ref{sec:pseudo.real}, however, instead of having the pseudomode localised around $0$, we work around a $\lambda$-dependent point. 

As the support of the pseudomode will be contained in~$\Real_+$, 
this construction is suitable also for operators in $L^2(\Real_+)$. 
In fact we shall rather proceed reversely and formulate the strategy for such a situation, the subsequent applicability of the results for problems in $L^2(\Real)$ 
is obvious. 

\subsection{Admissible class of potentials and curves}
\label{subsec:gc.pot.curv}
To keep the previous strategy working without more complicated and implicit conditions on $V$, we add an additional condition on $\Im V$, namely a control of $\Im V'(x)$. In detail, we assume the following.
\begin{Assumption}
	\label{asm:basic.gen}
	Let $N\in \Nat$, $N>1$, let $V \in W^{N,\infty}_{\rm loc}(\overline{\Real_+})$ satisfy
	\begin{equation}\label{asm.eq.ImV.gen}
	\lim_{x \to +\infty} \Im V(x) = + \infty
	\end{equation}
	together with all the conditions of 
	Assumption~\ref{asm:basic} for $x>0$. In addition suppose that 
	\begin{equation}\label{asm.eq:ImV'.lb}
	\forall x \gs 1, \qquad \Im V'(x)  \gs  \Im V(x) \langle x \rangle^\nu, 
	\end{equation}
	where $\nu:=\nu_+$. 
\end{Assumption}

In this section, we write
\begin{equation}\label{lam.ab.def}
\lambda = a + i b, \qquad a \in \Real, \ b \in \Real_+.
\end{equation} 
For sufficiently large $b$ we define 
the \emph{turning point}~$x_b$ of $\Im V$ by the equation
\begin{equation}\label{xb.def}
\Im V(x_b) = b,
\end{equation}
which is well-defined due to~\eqref{asm.eq:ImV'.lb}.
The cut-off is taken around the turning point $x_b$, namely:
\begin{equation}\label{xi.def.gen}
\begin{aligned}
& \xi \in C_0^\infty(\Real_+), \quad 0 \leq \xi \leq 1, 
\\
&\forall x \in (x_b-\delta+\Delta, x_b + \delta-\Delta), \quad \xi(x) =1,  
\\
\qquad 
& \forall x \notin (x_b-\delta,x_b + \delta), \quad \xi(x)=0. 
\end{aligned}
\end{equation}
Here we take
\begin{equation}\label{delta.def.gen}
\delta: =\frac{x_b^{- \nu}}2, \qquad \Delta := \frac{\delta}{4},
\end{equation}
and denote
\begin{equation}\label{Jb.def}
\cJ_b:=(x_b-\delta,x_b + \delta), \qquad
\cJ_b':=(x_b-\delta + \Delta, x_b + \delta -\Delta ).	
\end{equation}
Finally, we restrict the real part of $\lambda$ by
\begin{equation}\label{a.gen}
\forall x \in \cJ_b, \qquad  
b^\frac 23 x_b^\frac{2\nu}3 \ls |a|  \ls a- \Re V(x) \ls b^2 x_b^{-4\nu -4 \eps_1-2}.
\end{equation}
The set of admissible $a$'s is non-empty since $\sup_{x \in \cJ_b} |\Re V(x)| \ls b^2 x_b^{-4\nu -4 \eps_1-2}$ by assumption~\eqref{asm.eq:ReV} and the choice of $\cJ_b$ in \eqref{Jb.def}; moreover it follows from \eqref{asm.eq:ReV} that  $b^\frac 23 x_b^\frac{2\nu}3 \ls b^2 x_b^{-4\nu -4 \eps_1-2}$ for every sufficiently small $\eps_1>0$.

\subsection{Pseudomode construction}
The pseudomode will have the form
\begin{equation}\label{mode.def.gen}
f(x) := \xi(x) g(x) 
\qquad\mbox{with}\qquad
g(x) :=
\exp \left(- \sum_{k=-1}^{n-1} \lambda^{-\frac k2} \int_{x_b}^x   \psi_{k}'(t) \, \der t \right),
\end{equation}
where $\{\psi'_k\}_{k\in[[-1,n-1]]}$ are determined by \eqref{ODE}.

\begin{Proposition}\label{prop:cut.gen}
	Let Assumption~\ref{asm:basic.gen} hold, $0 \leq n \leq N$, $\{\psi'_k\}_{k\in[[-1,n-1]]}$ be determined by \eqref{ODE}, $\cJ_b$, $\cJ_b'$ be as in \eqref{Jb.def}, $\xi$, $g$ be as in \eqref{xi.def.gen}, \eqref{mode.def.gen}, respectively,  and $a$ satisfy \eqref{a.gen}.
	Then there exists $c>0$ such that
	\begin{equation}\label{cut.rates.gen}
	\frac{\|\xi'' g\|_{L^2(\Real_+)} + \|\xi' g'\|_{L^2(\Real_+)}}{\|\xi g\|_{L^2(\Real_+)}} 
	= \BigO(\exp(-c x_b^{\nu+1 +2\eps_1})), \quad b \to + \infty.
	\end{equation}
\end{Proposition}
\begin{proof}
Let us first estimate 
$\sgn(x -x_b) \int_{x_b}^x \Re (\lambda^\frac 12 \psi_{-1}'(t)) \; \der t$. 
For $x_b < x <  x_b+\delta$ (the other case is analogous), 
an analogue of the the complex square root formula~\eqref{sqr.rt}, 
the choice of~$a$ in~\eqref{a.gen} 
and the mean value theorem lead to 
\begin{equation}\label{Re.psi-1.gen}
\begin{aligned}
\Re (\lambda^\frac 12 \psi_{-1}'(x))
& 
\gs
\frac{\Im V(x) - b}{(a-\Re V(x))^\frac12 + (\Im V(x)-b)^\frac 12} 
\\
& \gs
\frac{\Im V'(x_b)(x-x_b)}{|a|^\frac12 + (\Im V'(x_b)(x-x_b))^\frac12}. 
\end{aligned}
\end{equation}
In the second inequality
we have used also that the values of $\Im V'$ at $\cJ_b$ are comparable, 
see \eqref{asm.eq:V'} and \eqref{V.x.Delta}. 
Hence, for every $ x \in \cJ_b \setminus \cJ_b'$, 
we have  
\begin{equation}\label{general.-1est}
\int_{x_b}^x\Re (\lambda^\frac 12 \psi_{-1}'(t)) \, \der t
\gs
\frac{b \, x_b^{-\nu}}{|a|^\frac12 + b^\frac 12}
\gs 
\begin{cases}
x_b^{\nu+1 +2\eps_1}, & |a|>b,
\\[1mm]
b^\frac12 \,  x_b^{-\nu}, & |a| \leq b.
\end{cases}
\end{equation}
Here the first inequality employs~\eqref{asm.eq:ImV'.lb}
in the numerator and~\eqref{asm.eq:V'} in the denominator,
while the second inequality follows from~\eqref{a.gen}.
Notice that by \eqref{asm.eq:ReV} we have 
$b^\frac12 \,  x_b^{-\nu} \gs x_b^{\nu+1 +2\eps_1}$,
so the left hand side tends to infinity
as $b \to +\infty$ too.

Next we investigate $\int_{x_b}^x|\lambda^{-\frac k2} \psi_k'|$ for $k \in [[0,n-1]]$ and $x \in \cJ_b$.
The estimates analogous to \eqref{psik'.exp}, \eqref{psik'.est} 
and the choice of~$a$ in~\eqref{a.gen} yield
\begin{equation}\label{|psi_k|.gen}
\begin{aligned}
\int_{x_b}^x
|\lambda^{-\frac k2} \psi_k'(t)| \, \der t 
&\ls 
\sum_{j=1}^{k+1}
\int_{x_b}^x
\frac{|V(t)|^j x_b^{(k+1)\nu}}{ |a-\Re V(t)|^{j + \frac k2}}
\, \der t 
\ls 
\sum_{j=1}^{k+1}
\frac{(|a|^j + b^j)x_b^{k \nu}}{ |a|^{j + \frac k2}}.
\end{aligned}
\end{equation}
Further from \eqref{a.gen} and \eqref{asm.eq:ReV}
\begin{equation}
\sum_{j=1}^{k+1}
\frac{|a|^j x_b^{k \nu}}{ |a|^{j + \frac k2}}
\ls
\left(\frac{x_b^{\nu}}{ |a|^\frac1 2}\right)^k
\ls
\left(\frac{x_b^{2\nu}}{b}\right)^\frac k3
\ls
x_b^{-\frac2 3k (\nu +1 + \eps_1)}
\end{equation}
and
\begin{equation}
\sum_{j=1}^{k+1}
\frac{b^j x_b^{k \nu}}{ |a|^{j + \frac k2}}
\ls
\begin{cases}
x_b^{-k(\nu+1+\eps_1)}, & |a|>b,
\\[1mm]
\sum_{j=1}^{k+1} \left(b x_b^{-2\nu} \right)^\frac{j-k}3, & |a| \leq b.
\end{cases}
\end{equation}
Thus, using again \eqref{asm.eq:ReV}, 
for every $x \in \cJ_b \setminus \cJ_b'$ we get
\begin{equation}\label{psi.k.psi.1.gen}
\frac{\int_{x_b}^x|\lambda^{-\frac k2} \psi_k'(t)| \, \der t}{\int_{x_b}^x\Re (\lambda^\frac 12 \psi_{-1}'(t)) \, \der t}
\ls 
x_b^{-(\frac2 3k+1) (\nu +1 + \eps_1)}
+
\begin{cases}
x_b^{-(k+1)(\nu+1+\eps_1)}, & |a|>b,
\\[1mm]
x_b^{-(\nu+1+\eps_1)}, & |a| \leq b.
\end{cases}
\end{equation}

Using~\eqref{general.-1est} with help of~\eqref{psi.k.psi.1.gen}
and~\eqref{xi.der},
we obtain (with some $C_1>0$)
\begin{equation}\label{xi'g.gen}
\|\xi'' g\|_{L^2(\Real_+)} + \|\xi' g'\|_{L^2(\Real_+)} 
\ls 
\exp \left(- C_1 \frac{b x_b^{-\nu}}{|a|^\frac12 + b^\frac 12} \right).
\end{equation}
The estimate is clear for the first norm on the left hand side.
To control the extra terms obtained by differentiating~$g$,
we employ the bounds coming from Gronwall's inequality~\eqref{Gronw.est}
for the term $\lambda^\frac{1}{2} \psi_{-1}'$
and the other terms $\lambda^{-\frac{k}{2}} \psi_k'$ 
can be estimated similarly as in~\eqref{|psi_k|.gen}.

Finally, to show \eqref{cut.rates.gen}, we need to verify that $\|\xi g\|_{L^2(\Real_+)}$ is not too small. To this end notice that for $a< b$
\begin{equation}\label{gen.norm.1}
\begin{aligned}
\int_{x_b}^{x_b+x_b^{-2 |\nu|}} |\Re (\lambda^\frac 12 \psi_{-1}'(t))| \; \der t 
&\ls 
\int_{x_b}^{x_b+x_b^{-2 |\nu|}} |\Im V(t) - b|^\frac12 \; \der t 
\\
&\ls 
b^\frac12 x_b^{-3|\nu| + \frac12 \nu}
\end{aligned}
\end{equation}
and for $a \geq b$
\begin{equation}\label{gen.norm.2}
\begin{aligned}
\int_{x_b}^{x_b+x_b^{-2 |\nu|}} |\Re (\lambda^\frac 12 \psi_{-1}'(t))| \; \der t 
&\ls 
\int_{x_b}^{x_b+x_b^{-2 |\nu|}} \frac{\Im V'(x_b)(t-x_b)}{|a|^\frac 12} \; \der t 
\\
&\ls
|a|^{-\frac12} b \, x_b^{\nu-4|\nu|}.
\end{aligned}
\end{equation}
Since
\begin{equation}
\begin{cases}
\displaystyle 
b^\frac12 x_b^{-3|\nu| + \frac12 \nu} = o \left(\frac{b \, x_b^{-\nu}}{|a|^\frac12 + b^\frac 12}\right), & a<b,
\\[1mm]
\displaystyle 
|a|^{-\frac12} b \, x_b^{\nu-4|\nu|} = o \left(\frac{b \, x_b^{-\nu}}{|a|^\frac12 + b^\frac 12}\right), & a \geq b,
\end{cases}
\end{equation}
we obtain in both cases (with some $C_2>0$)
\begin{equation}\label{xi'g.gen.2}
\frac{\|\xi'' g\|_{L^2(\Real_+)} + \|\xi' g'\|_{L^2(\Real_+)}}{\|\xi g\|_{L^2(\Real_+)}}
\ls 
\exp \left(- C_2 \frac{b \, x_b^{-\nu}}{|a|^\frac12 + b^\frac 12} \right).
\end{equation}
The claim \eqref{cut.rates.gen} follows from the last inequality in \eqref{general.-1est}.
\end{proof}
\begin{Theorem}\label{thm:basic.gen}
Let Assumption~\ref{asm:basic.gen} hold, $f$ be as in \eqref{mode.def.gen} with $n=N-1$ and $a$ satisfy \eqref{a.gen}.
Then, as $b \to + \infty$, 
\begin{equation}\label{HV.f.dec.gen}
\begin{aligned}
\frac{\|(H_V-\lambda) f\|_{L^2(\Real_+)}}{\|f\|_{L^2(\Real_+)}} 
&= 
\BigO \left( \exp(-c x_b^{\nu+1 +2\eps_1})
+
x_b^{N\nu}
\sup_{x \in \cJ_b} \frac{b + |\Re V(x)|}{\left(a-\Re V(x)\right)^\frac{N}{2}} 
\right.
\\
& \quad  \left. +
\sum_{k=0}^{N-2}
\sum_{l=2}^{N+k}
x_b^{(N+k)\nu}
\sup_{x \in \cJ_b} \frac{(b + |\Re V(x)|)^l}{\left(a-\Re V(x)\right)^{\frac{N-2+k}{2}+l}} 
\right).
\end{aligned}
\end{equation}
\end{Theorem}
\begin{proof}
The claim follows straightforwardly from Proposition~\ref{prop:cut.gen}, the estimate of the remainder $|r_n|$, see \eqref{rn.est}, and the choice of $a$, see \eqref{a.gen}. 
\end{proof}

\subsection{Examples}
\label{subsec:Ex.gen}
\begin{Example}[Example \ref{ex:pol.1} continued] 	
\label{ex:pol.Omega}
	
We illustrate applicability of Theorem~\ref{thm:basic.gen} on the imaginary  monomial potentials, namely $V(x) = i x^\gamma$ for $x>0$ and $\gamma \geq 1$. With this choice, we have $\nu=-1$, $x_b = b^\frac 1\gamma$ and we may take $a$ as (with $\eps>0$)
\begin{equation}\label{optimality}
b^{\frac 23 \frac{\gamma-1}{\gamma} + \eps } 
\ls a \ls 
b^{2 \frac{\gamma+1}{\gamma} - \eps},
\end{equation}
see \eqref{a.gen}. Straightforward calculations yield that for a sufficiently large $N$ we get a decay in \eqref{HV.f.dec.gen}. In other words we show that there are pseudomodes (with a decay in \eqref{HV.f.dec.gen}) in a region bounded by curves $\Gamma_{1,2}$ in $\Com$ given by
\begin{equation}
\Gamma_1(t) := t^{\frac 23 \frac{\gamma-1}{\gamma} + \eps} + i t, \qquad 
\Gamma_2(t) := t^{2 \frac{\gamma+1}{\gamma} - \eps} + it. 
\end{equation}
Notice that for $\gamma =2$, we obtain (with an obvious re-parametrisation) the curves $\eta + i \eta^p$ with $1/3 < p<3$ of the Boulton's conjecture, \cf~\cite{Boulton_2002},  which are known to be optimal, \cf~\cite{Pravda-Starov_2006}. 
\end{Example}

\begin{Example}[Semiclassical operators]\label{ex:s-c}
Let us briefly explain how the semiclassical setting, see \eg~\cite{Davies_1999-NSA}, can be treated using our approach and how previously used assumptions can be relaxed. For a sufficiently regular potential~$U$, we search for pseudomodes 
of the semiclassical operator
\begin{equation}
- h^2 \frac{\der^2}{\der x^2} + U(x) -z, \quad h >0,
\end{equation}
corresponding to a pseudoeigenvalue $z \in \Com$,
in the limit $h \to 0$. 

First we factor the parameter $h^2$ out and obtain \eqref{Hamiltonian.intro}
with the scaled potential $V(x) := h^{-2} U(x)$ 
and pseudoeigenvalue $\lambda := h^{-2} z$ in our notations,
see~\eqref{objective.intro}. 
The pseudomode is constructed around the point $x_0$ satisfying the equation $\Im V(x_0) = \Im \lambda$, \ie~$\Im U(x_0) = \Im z$. Notice that $x_0$ is determined only by $\Im z$, which is fixed here. 

The cut-off is successful if there exist $\delta_\pm$ such that for all $x \in (x_0 + \delta_+/2,x_0 + \delta_+)$
\begin{equation}\label{Re.psi-1.s-c}
\int_{x_0}^{x} \frac{\Im U(t) - \Im U(x_0)}{(\Re z-\Re U(t))^\frac12 
+ |\Im U(t) - \Im U(x_0)|^\frac 12} \, \der t
\\
\gs  h^{1-\eps}
\end{equation}
with some $\eps >0$ and similarly for $\delta_-$. Indeed, an appropriately modified first inequality in \eqref{Re.psi-1.gen} yields
\begin{equation}
\Re (\lambda^\frac 12 \psi_{-1}'(t)) \gs h^{-1} \frac{\Im U(t) 
- \Im U(x_0)}{(\Re z-\Re W(t))^\frac12 + |\Im U(t) - \Im U(x_0)|^\frac 12}.
\end{equation}
However, \eqref{Re.psi-1.s-c} can be satisfied \eg~when the
Davies' condition \cite{Davies_1999-NSA}
\begin{equation}\label{Davies}
  \Im U'(x_0) >0
  \qquad \mbox{and} \qquad
  z = \eta^2 + U(x_0) 
  \quad \mbox{with} \quad \eta^2>0
\end{equation}
is imposed; indeed, Taylor's theorem yields 
\begin{equation}
\begin{aligned}
\Im U(t) - \Im U(x_0) &= \Im U'(x_0)(t-x_0) + \BigO((t-x_0)^2),&
\\
\Re z - \Re U(t) &= \eta^2+ \BigO(|t-x_0|), &t \to x_0,
\end{aligned}
\end{equation}
and so the choice $\delta_+ := h^\frac{1-\eps}2$ works.
It can be easily checked that the other terms in the expansion are harmless. Finally, the decay of the remainders $r_n$ follows easily if $|\Re z - \Re U(x)|$ is not too small 
on $(x_0-\delta_-,x_0+\delta_+)$, 
which is satisfied when the
Davies' condition~\eqref{Davies} holds; 
as an illustration, we have
\begin{equation}
h^2 |r_0| \ls \frac{h}{|\Re z - \Re U(x)|^\frac 12}
\end{equation}
for all $x \in (x_0-\delta_-,x_0+\delta_+)$.

In summary, the semiclassical setting allows for many simplifications and a suitable behaviour of~$U$ around a fixed point $x_0$ only is needed to obtain pseudomodes (localising around $x_0$) as $h \to 0$. It is also clear that the previously used conditions of 
the type $\Im U'(x_0)\neq 0$ are not needed as we may use a larger neighbourhood of~$x_0$ and take a sufficiently large $\eta$ to satisfy \eqref{Re.psi-1.s-c} and obtain a decay of~$r_n$. 
\end{Example}

\begin{Example}[Strong local singularities]\label{ex:sing}
In all previous pseudomode constructions, 
we used the behaviour of the potential~$V$ at infinity. 
If~$V$ is sufficiently singular at a finite point, 
the construction of the present Section~\ref{sec:pseudo.gen} 
can be adapted accordingly.
We illustrate this on an example in $L^2(\Real_-)$ with
\begin{equation}\label{V.sing}
  V(x) := \frac{i}{|x|^\alpha}
  \qquad \text{for} \qquad x\in (-1,0), 
  \qquad \alpha >2, 
\end{equation}
and an arbitrary behaviour outside $(-1,0)$. We consider $\Real_-$ for convenience only so that \eqref{asm.eq.ImV.gen} holds  for $x \to 0-$ and the shape of already derived formulas is preserved.
Considering~$\Real_+$ instead of~$\Real_-$ and
further generalisations 
in the sense of Section~\ref{subsec:gc.pot.curv}
(like $\Re V \neq 0$ or $\nu > -1$) 
are straightforward. 
We emphasise in particular the potentials with 
$\Re V(x)=c/|x|^2$, $c \in\Real$, 
appearing in the radial part of 
higher-dimensional Schr\"odinger operators.

We follow the notations of Section~\ref{subsec:gc.pot.curv} 
and construct a pseudomode of the type~\eqref{mode.def.gen} 
around the turning point~$x_b$ of~$\Im V$
that tends to~$0-$ as $b \to +\infty$.  
In more detail, we take here
\begin{equation}
\begin{aligned}
\lambda &= a + ib, \qquad a,b \in \Real_+,
\\
\Im V(x_b)&=b, 
\qquad \delta:=\frac{|x_b|}{2}, 
\quad \Delta:= \frac \delta 4,
\end{aligned}
\end{equation}
with~$\delta$ going to zero as $b\to\infty$,
and the cut-off $\xi$  as well as intervals $\cJ_b$ and $\cJ_b'$ are as in \eqref{xi.def.gen}, \eqref{Jb.def}, respectively. The new condition on admissible $a$'s 
(corresponding to the simple case \eqref{V.sing}) reads
\begin{equation}\label{a.gen.sing}
b^{\frac 23(1+\frac 1\alpha)} \ls a  \ls b^{2(1-\frac 1\alpha)-\epsilon}
\end{equation}
with some $\epsilon>0$.

Following and slightly adapting the estimates in the proof of Proposition~\ref{prop:cut.gen}, we get for every $x \in \cJ_b \setminus \cJ_b'$ with $x>x_b$ that 
\begin{equation}\label{general.-1est.sing}
\int_{x_b}^x\Re (\lambda^\frac 12 \psi_{-1}'(t)) \, \der t
\gs
\frac{b^{1-\frac 1\alpha}}{a^\frac12 + b^\frac 12}.
\end{equation}
Here the importance of the assumed condition $\alpha>2$, 
as well as of \eqref{a.gen.sing}, 
is clearly visible in order to ensure that the right-hand side tends to infinity as $b\to+\infty$. 
Further, it can be straightforwardly checked  that the cut-off is indeed successful and an analogue of \eqref{xi'g.gen.2} holds; we remark that in the estimates like 
\eqref{gen.norm.1} and \eqref{gen.norm.2} we integrate \eg~over $(x_b,x_b+x_b^2)$.

The remainder estimate is also straightforward, using \eqref{rn.est}, we obtain altogether that with $V$ as in \eqref{V.sing} there exists a positive constant~$c$ such that
\begin{equation}\label{HV.f.dec.sing}
\frac{\|(H_V-\lambda) f\|_{L^2(\Real_-)}}{\|f\|_{L^2(\Real_-)}} 
= 
\BigO \left( \exp(-c b^\frac \epsilon 2 )
+ 
\frac{b^{1+\frac{n+1}{\alpha}}}{a^\frac{n+1}{2}} + \sum_{k=0}^{n-1}\sum_{l=2}^{n+1+k} \frac{b^{l+\frac{n+1+k}{\alpha}}}{a^{l+\frac{n-1+k}{2}}}
\right)
\end{equation}
as $b \to + \infty$ 
(then necessarily also $a \to +\infty$ due to~\eqref{a.gen.sing}).
Similarly as in Example~\ref{ex:pol.Omega}, 
we can check that if we strengthen~\eqref{a.gen.sing} to
\begin{equation}\label{a.gen.sing.dec}
b^{\frac 23(1+\frac 1\alpha)+\epsilon} \ls a  \ls b^{2(1-\frac 1\alpha)-\epsilon}
\end{equation}
with some $\epsilon>0$, then for a sufficiently large $n$ we indeed have a decay in \eqref{HV.f.dec.sing}.
\end{Example}


\appendix

\section{Proofs of Lemmata~\ref{lem:str} and \ref{lem:r_n.est}}

In the following, notations of Lemmata~\ref{lem:str} and \ref{lem:r_n.est} are used and $V$ is assumed to be sufficiently regular so that all appearing derivatives of it exist. 

In the first step, we investigate certain operations on $T_j^{r,s}$, defined in \eqref{Tjr.def}. To simplify notations, we view $T_j^{r,s}$ as a set of functions of the prescribed form with all possible choices of constants $c_\alpha$. We start with the following simple observations. 

\begin{enumerate}[(a)]
\item\label{remarkinduction(i)} If $r \geq 1$, then $\mathcal{I}_0^{r,r+1} = \emptyset$ and so $T_0^{r,r+1}= \{0\}.$
\item\label{remarkinduction(ii)} $T_0^{0,1} = \Com$  since $ \mathcal{I}_0^{0,1}=\{0\}. $  
\item\label{remarkinduction(iii)} $ T_{j}^{r,s}+T_{j}^{r,s}=T_{j}^{r,s}$ and so $\sum_{j=0}^{s}T_{j}^{r,s-j}+\sum_{j=0}^{s}T_{j}^{r,s-j}=\sum_{j=0}^{s}T_{j}^{r,s-j}. $
\item\label{remarkinduction(iv)} For $s_1\leq s_2$, $T_{j}^{r,s_1}  \subset T_{j}^{r,s_2}$ since taking $\alpha_i=0$ is allowed.
\item\label{remarkinduction(v)} $c \, T_{j}^{r,s}=T_{j}^{r,s} $ 
for any constant $c\in \mathbb{C}$.
\end{enumerate}

\begin{Lemma}\label{lem:T.op} 
Let $T_j^{r,s}$ be as in \eqref{Tjr.def}. Then
\begin{align}
\left(T_j^{r,s}\right)' &\subset T_j^{r+1,s+1},
\label{eq:T'}
\\
V'T_j^{r,s} & \subset T_{j+1}^{r+1,s},
\label{eq:V'T}
\\
T_{j_{1}}^{r_{1},s_{1}}T_{j_{2}}^{r_{2},s_{2}} &\subset T_{j_{1}+j_{2}}^{r_{1}+r_{2},\max\{s_{1},s_{2}\}} \subset T_{j_{1}+j_{2}}^{r_{1}+r_{2},s_{1}+s_{2}-1}.
\label{eq:T1T2}
\end{align}

\end{Lemma}
\begin{proof}
To prove \eqref{eq:T'}, the product rule yields
	\begin{equation}
	\begin{aligned}
	&\left((V^{(1)})^{\alpha_1}(V^{(2)})^{\alpha_2+1} \dots (V^{(s)})^{\alpha_s} \right)'
	\\&\quad =  
	\alpha_1 (V^{(1)})^{\alpha_1-1}(V^{(2)})^{\alpha_2+1} \dots (V^{(s)})^{\alpha_s}  
	\\ & \qquad 
	+  \alpha_2 (V^{(1)})^{\alpha_1}(V^{(2)})^{\alpha_2-1}(V^{(3)})^{\alpha_3+1} \dots (V^{(s)}) 
	\\
	& \qquad   + \dots + \alpha_s (V^{(1)})^{\alpha_1} \dots (V^{(s)})^{\alpha_s-1}V^{(s+1)}.
	\end{aligned}
	\end{equation}
	Since $\sum_{i=1}^{s} \alpha_i=j$ and $\sum_{i=1}^{s} i\alpha_i=r$, we obtain that the sum of powers in every term is again $j$, \eg~
\begin{equation}
\alpha_{1}-1+\alpha_{2}+1+\dots+\alpha_{s}= \alpha_{1}+\alpha_{2}+\dots+\alpha_{s}=j,
\end{equation}
and the sum of powers multiplied by the order of the derivative is $r+1$, \eg~
\begin{equation}
1 (\alpha_{1}-1)+ 2 (\alpha_{2}+1)+\dots +s \alpha_{s} 
= 1+ 1 \alpha_{1}+2 \alpha_{2}+\dots+s \alpha_{s}=r+1.
\end{equation}
So \eqref{eq:T'} follows by the rule \ref{remarkinduction(iii)}.

To verify \eqref{eq:V'T}, observe that
\begin{equation}
V' (V^{(1)})^{\alpha_1}(V^{(2)})^{\alpha_2} \dots (V^{(s)})^{\alpha_s}
= (V^{(1)})^{\alpha_1+1}(V^{(2)})^{\alpha_2} \dots (V^{(s)})^{\alpha_s},
\end{equation}
and
\begin{equation}
\begin{aligned}
(\alpha_{1}+1)+\alpha_{2}+\dots+\alpha_{s}&=j+1,
\\
1 (\alpha_{1}+1) +2  \alpha_{2}+\dots+s \alpha_{s}
&= 1 \alpha_{1} +2 \alpha_{2}+\dots+s \alpha_{s}=s+1.
\end{aligned}
\end{equation}

To show \eqref{eq:T1T2}, we assume without the loss of generality 
that $s_1 \leq s_2 $. 
Hence
\begin{equation}
\begin{aligned}
&\sum_{\alpha\in \mathcal{I}_{j_{1}}^{r_{1},s_{1}}} c_{\alpha}(V^{(1)})^{\alpha_1}(V^{(2)})^{\alpha_2} \dots (V^{(s_1)})^{\alpha_{s_1}} \sum_{\beta\in \mathcal{I}_{j_{2}}^{r_{2},s_{2}}} c_{\beta}(V^{(1)})^{\beta_1}(V^{(2)})^{\beta_2} \dots (V^{(s_2)})^{\beta_{s_2}} &
\\& = \sum_{\substack{\alpha\in \mathcal{I}_{j_{1}}^{r_{1},s_{1}}, \\ \beta\in \mathcal{I}_{j_{2}}^{r_{2},s_{2}} }} c_{\alpha,\beta}(V^{(1)})^{\alpha_1+\beta_1}(V^{(2)})^{\alpha_2+\beta_2}\dots (V^{(s_1)})^{\alpha_{s_1}+\beta_{s_1}}(V^{(s_{1}+1)})^{\beta_{{s_1}+1}}\dots (V^{(s_2)})^{\beta_{s_2}}.
\end{aligned}
\end{equation}
For the powers in the resulting sum, we have 
\begin{equation}
\begin{aligned}
&1 (\alpha_1+\beta_1)+\dots+s_1  (\alpha_{s_1}+\beta_{s_1})+(s_{1}+1) \beta_{{s_1}+1}+\dots+s_{2} \beta_{s_2} \\
& \quad = 1 \alpha_1+\dots+s_{1} \alpha_{s_1}+1 \beta_1+\dots+s_{2} \beta_{s_2}=r_1+r_2,
\\
&(\alpha_1+\beta_1)+\dots+(\alpha_{s_1}+\beta_{s_1})+\beta_{{s_1}+1}+\dots+\beta_{s_2}
\\
& \quad =\alpha_1+\dots+\alpha_{s_1}+\beta_1+\dots+\beta_{s_2}=j_1+j_2,
\end{aligned}
\end{equation}
hence, 
\begin{equation} \label{maxs1s2}
T_{j_{1}}^{r_{1},s_{1}}T_{j_{2}}^{r_{2},s_{2}} 
\subset T_{j_{1}+j_{2}}^{r_{1}+r_{2},\max\{s_{1},s_{2}\}}.
\end{equation}
Since  $s_1 \geq 1$ and $ s_2 \geq 1$, we have $\max\{s_1,s_2\} \leq s_1+s_2-1$ and thus the rule \ref{remarkinduction(iv)} yields the second inclusion in \eqref{eq:T1T2}.
\end{proof}

In the next step, we find the form of the derivatives of $\psi_{-1}'$ and $1/\psi_{-1}'$.
\begin{Lemma}\label{lem:psi-1'.ind}
Let $\psi_{-1}'$ be as in \eqref{ODE} and $T_j^{r,s}$ be defined as in \eqref{Tjr.def}. Then, for every $m \in \Nat_0$, we have
\begin{align}
\psi_{-1}^{(m+1)} & \in \frac{\lambda^{-\frac{1}{2}}}{(\lambda-V)^{-\frac{1}{2}}}\sum_{j=0}^{m}\dfrac{T_j^{m,m+1-j}}{(\lambda-V)^j},
\label{psi.m}
\\
\left(\dfrac{1}{\psi_{-1}'}\right)^{(m)} & \in	\dfrac{\lambda^{\frac{1}{2}}}{(\lambda-V)^{\frac{1}{2}}}\sum_{j=0}^{m}\dfrac{T_{j}^{m,m+1-j}}{(\lambda-V)^{j}}.
\label{psi.-.m}
\end{align}
\end{Lemma}
\begin{proof}
We give the detailed induction proof with respect to $m \in \Nat_0$
for \eqref{psi.m} only, the proof of \eqref{psi.-.m} is fully analogous. Recall that $\psi_{-1}'= i \lambda^{-\frac 12} (\lambda-V)^{\frac{1}{2}}$ and so the base step can be easily verified in both cases using the rule \ref{remarkinduction(ii)}.

We make the induction step $m \rightarrow m+1$, \ie~we assume that the formula for $\psi_{-1}^{(m)}$ holds. Using the rule \ref{remarkinduction(iii)} and always incorporating the constants arising by taking the derivatives in the constants in the sets $T$, we arrive at
\begin{align} 
\psi_{-1}^{(m+1)}&=\left(\psi_{-1}^{(m)} \right)' \in \left(
\frac{\lambda^{-\frac{1}{2}}}{(\lambda-V)^{-\frac{1}{2}}} \sum_{j=0}^{m-1}\dfrac{T_j^{m-1,m-j}}{(\lambda-V)^j} 
\right)'
\\
&=
\left(
\frac{\lambda^{-\frac{1}{2}}}{(\lambda-V)^{-\frac{1}{2}}}\right)' 
\sum_{j=0}^{m-1}\frac{T_j^{m-1,m-j}}{(\lambda-V)^j} +  \frac{\lambda^{-\frac{1}{2}}}{(\lambda-V)^{-\frac{1}{2}}} \sum_{j=0}^{m-1}
\left(
\frac{T_j^{m-1,m-j}}{(\lambda-V)^j} 
\right)'
\\&= 
\frac{\lambda^{-\frac{1}{2}}V'}{(\lambda-V)^{\frac{1}{2}}} \sum_{j=0}^{m-1} \frac{T_j^{m-1,m-j}}{(\lambda-V)^j} + \frac{\lambda^{-\frac{1}{2}}}{(\lambda-V)^{-\frac{1}{2}}} \sum_{j=0}^{m-1}
\left(
\frac{\left(T_j^{m-1,m-j}\right)'}{(\lambda-V)^{j}} + \frac{V'T_j^{m-1,m-j}}{(\lambda-V)^{j+1}}
\right)
\\[-5mm]
&= 
\frac{\lambda^{-\frac{1}{2}}}{(\lambda-V)^{-\frac{1}{2}}} 
\sum_{j=0}^{m-1} \frac{V'T_j^{m-1,m-j}}{(\lambda-V)^{j+1}} + \frac{\lambda^{-\frac{1}{2}}}{(\lambda-V)^{-\frac{1}{2}}} 
\sum_{j=0}^{m-1} \frac{\left(T_j^{m-1,m-j}\right)'}{(\lambda-V)^{j}}.
\end{align}
Using the rule \ref{remarkinduction(i)}, we can let the first sum start with $j=-1$ and Lemma \ref{lem:T.op} yields 
\begin{align}
\begin{aligned} \psi_{-1}^{(m+1)}& \in \dfrac{\lambda^{-\frac{1}{2}}}{(\lambda-V)^{-\frac{1}{2}}}\sum_{j=-1}^{m-1}\dfrac{T_{j+1}^{m,m-j}}{(\lambda-V)^{j+1}}+ \dfrac{\lambda^{-\frac{1}{2}}}{(\lambda-V)^{-\frac{1}{2}}}\sum_{j=0}^{m-1}\dfrac{T_{j}^{m,m-j+1}}{(\lambda-V)^{j}} 
\\
&\subset 
\dfrac{\lambda^{-\frac{1}{2}}}{(\lambda-V)^{-\frac{1}{2}}}
\left(\sum_{j=0}^{m}\dfrac{T_{j}^{m,m-j+1}}{(\lambda-V)^{j}} 
+ \sum_{j=0}^{m-1}\dfrac{T_{j}^{m,m-j+1}}{(\lambda-V)^{j}}\right)
\\
&=
\dfrac{\lambda^{-\frac{1}{2}}}{(\lambda-V)^{-\frac{1}{2}}}
\sum_{j=0}^{m}\dfrac{T_{j}^{m,m-j+1}}{(\lambda-V)^{j}},
\end{aligned}
\end{align}
where we use the rule \ref{remarkinduction(iii)} in the last step. Thus \eqref{psi.m} is proved.
\end{proof}

Now we are ready to prove Lemmata~\ref{lem:str} and \ref{lem:r_n.est}. 

\begin{proof}[Proof of Lemma~\ref{lem:str}]
We give an induction argument 
with respect to the index $k\in[[-1,n-1]]$. 
To verify the base case $k=-1$ and $m\in[[1,n+2]]$, we apply \eqref{psi.m} from Lemma \ref{lem:psi-1'.ind} above. As the induction step we take $k\rightarrow k+1$, \ie~we start with assuming that for a fixed $k$ \eqref{psi.k.m} holds for all $m \in [[1,n+1-k]]$ and we have to show that
\begin{equation}
\psi_{k+1}^{(m)}
\in \frac{\lambda^{\frac{k+1}{2}}}{(\lambda-V)^{\frac{k+1}{2}}} \sum_{j=0}^{k+1+m} 
\frac{T_j^{k+1+m,k+m+2-j}}{(\lambda-V)^j}, \qquad  m\in[[1,n-k]].
\end{equation}

By formula \eqref{ODE}, we have 
\begin{equation} 
\begin{aligned}
\psi_{k+1}^{(m)}=(\psi_{k+1}')^{(m-1)}&=\Bigg[\frac{1}{2\psi_{-1}'}\Bigg(\psi_{k}'' - \sum_{\substack{\omega+\chi=k \\ \omega,\chi\not=-1}}^{}\psi_{\omega}'\psi_{\chi}'\Bigg)\Bigg]^{(m-1)}
\\& =\Bigg(\frac{1}{2\psi_{-1}'}\psi_{k}''\Bigg)^{(m-1)}-\Bigg(\frac{1}{2\psi_{-1}'}\sum_{\substack{\omega+\chi=k \\ \omega,\chi\not=-1}}^{}\psi_{\omega}'\psi_{\chi}'\Bigg)^{(m-1)}.
\end{aligned}	
\end{equation}
Using the Leibniz product rule once for the first term and twice for the second one, we arrive at
\begin{equation}\label{Leibniz}
\begin{aligned}
\psi_{k+1}^{(m)}=& \sum_{i=0}^{m-1} \binom{m-1}{i}
\Big(\frac{1}{2\psi_{-1}'}\Big)^{(i)}\psi_{k}^{(m+1-i)}
\\
&-\sum_{i=0}^{m-1} \Bigg[
\binom{m-1}{i}\Big(\dfrac{1}{2\psi_{-1}'}\Big)^{(i)}\sum_{\substack{\omega+\chi=k \\ \omega,\chi\not=-1}}^{} \sum_{p=0}^{m-1-i}\binom{m-1-i}{p} \psi_{\omega}^{(p+1)}\psi_{\chi}^{(m-i-p)}
\Bigg].
\end{aligned} 
\end{equation}

Let us now have a more detailed look at the term $ (\psi_{\omega})^{(p+1)}(\psi_{\chi})^{(m-i-p)}$.
We notice first that $\omega+\chi=k$, $ \omega,\chi\not=-1 $ and $(p+1)$, $(m-i-p) \in [[1,n-k]]$, which enables us to use the induction assumption to rewrite both $ \psi_{\omega}^{(p+1)}$ and $\psi_{\chi}^{(m-i-p)}$. Namely, the formula for the product of series yields
\begin{equation}
\begin{aligned}
&\psi_{\omega}^{(p+1)}\psi_{\chi}^{(m-i-p)} 
\\
& \in
\frac{\lambda^{\frac{\omega}{2}}}{(\lambda-V)^{\frac{\omega}{2}}} \sum_{j_{1}=0}^{\omega+p+1} \frac{T_{j_{1}}^{\omega+p+1,\omega+p+2-j_{1}}}{(\lambda-V)^{j_{1}}}
\frac{\lambda^{\frac{\chi}{2}}}{(\lambda-V)^{\frac{\chi}{2}}} \sum_{j_{2}=0}^{\chi+m-i-p}\dfrac{T_{j_{2}}^{\chi+m-i-p,\chi+m-i-p+1-j_{2}}}{(\lambda-V)^{j_{2}}}
\\
&\subset\frac{\lambda^{\frac{\omega+\chi}{2}}}{(\lambda-V)^{\frac{\omega+\chi}{2}}}\sum_{j=0}^{\omega+\chi+m+1-i} \sum_{q=0}^{j}\frac{T_{q}^{\omega+p+1,\omega+p+2-q}}{(\lambda-V)^{q}}
\dfrac{T_{j-q}^{\chi+m-i-p,\chi+m-i-p+1-j+q}}{(\lambda-V)^{j-q}},
\end{aligned}
\end{equation}
where we set 
\begin{equation}\label{T.conv}
T_{l}^{m,n}=\emptyset, \quad  l \geq m.
\end{equation}

Using Lemma \ref{lem:T.op} and the fact that $\omega+\chi=k$, we obtain
\begin{equation}
\psi_{\omega}^{(p+1)}\psi_{\chi}^{(m-i-p)}
\in 
\frac{\lambda^{\frac{k}{2}}}{(\lambda-V)^{\frac{k}{2}}} 
\sum_{j=0}^{k+m+1-i} \sum_{q=0}^j
\frac{T_{j}^{k+m+1-i,k+m+2-i-j}}{(\lambda-V)^{j}},
\end{equation}
where the resulting terms no longer depend on $p$, $\omega$ or $\chi$. 
Coming back to~\eqref{Leibniz},
by Lemma \ref{lem:psi-1'.ind} and the induction assumption, we get further that
\begin{equation}
\begin{aligned}
\psi_{k+1}^{(m)} &\in \sum_{i=0}^{m-1} \frac{\lambda^{\frac{1}{2}}}{(\lambda-V)^{\frac{1}{2}}} \sum_{j_{3}=0}^{i} \frac{T_{j_{3}}^{i,i+1-j_{3}}}{(\lambda-V)^{j_{3}}}
\Bigg(
\frac{\lambda^{\frac{k}{2}}}{(\lambda-V)^{\frac{k}{2}}} \sum_{j_{4}=0}^{k+m+1-i} \frac{T_{j_{4}}^{k+m+1-i,k+m+2-i-j_{4}}}{(\lambda-V)^{j_{4}}}
\\
& \qquad  +  \sum_{\substack{\omega+\chi=k \\ \omega,\chi\not=-1}}^{}\sum_{p=0}^{m-1-i}\dfrac{\lambda^{\frac{k}{2}}}{(\lambda-V)^{\frac{k}{2}}}\sum_{j_{4}=0}^{k+m+1-i}\dfrac{T_{j_{4}}^{k+m+1-i,k+m+2-i-j_{4}}}{(\lambda-V)^{j_{4}}} 
\Bigg).
\end{aligned}
\end{equation}
By the rule \ref{remarkinduction(iii)}, the formula for the product of series (with the convention \eqref{T.conv}) and the rule~(e) in the last step, we obtain
\begin{equation}
\begin{aligned}
\psi_{k+1}^{(m)}
& \in \sum_{i=0}^{m-1}\frac{\lambda^{\frac{k+1}{2}}}{(\lambda-V)^{\frac{k+1}{2}}}\sum_{j_{3}=0}^{i}\dfrac{T_{j_{3}}^{i,i+1-j_{3}}}{(\lambda-V)^{j_{3}}}\sum_{j_{4}=0}^{k+m+1-i}\dfrac{T_{j_{4}}^{k+m+1-i,k+m+2-i-j_{4}}}{(\lambda-V)^{j_{4}}}
\\
&=
\sum_{i=0}^{m-1}\dfrac{\lambda^{\frac{k+1}{2}}}{(\lambda-V)^{\frac{k+1}{2}}}\sum_{j=0}^{k+m+1}\sum_{q=0}^{j}\dfrac{T_{q}^{i,i+1-q}}{(\lambda-V)^{q}}
\dfrac{T_{j-q}^{k+m+1-i,k+m+2-i-j+q}}{(\lambda-V)^{j-q}} 
\\
& \subset \sum_{i=0}^{m-1}\dfrac{\lambda^{\frac{k+1}{2}}}{(\lambda-V)^{\frac{k+1}{2}}}\sum_{j=0}^{k+m+1}\sum_{q=0}^{j}\dfrac{T_{j}^{k+m+1,k+m+2-j}}{(\lambda-V)^{j}}
\\
&=
\dfrac{\lambda^{\frac{k+1}{2}}}{(\lambda-V)^{\frac{k+1}{2}}}\sum_{j=0}^{k+m+1}\dfrac{T_{j}^{k+m+1,k+m+2-j}}{(\lambda-V)^{j}}.
\end{aligned}
\end{equation}
This concludes the proof of Lemma~\ref{lem:str}.
\end{proof}

\begin{proof}[Proof of Lemma~\ref{lem:r_n.est}]
Recalling~\eqref{binomial}, we can write the remainder $r_n$ in the following way
\begin{equation}\label{r_nmits1unds2}
\begin{aligned}r_n & = \sum_{k=n-1}^{2(n-1)}\lambda^{-\frac{k}{2}}\phi_{k+1} 
= \sum_{k=n-1}^{2(n-1)} \lambda^{-\frac{k}{2}} \Bigg(\psi_{k}''- \! \! \! \!  \sum_{\substack{\omega+\chi=k \\ -1\leq\omega,\chi\leq  n-1}}^{}\psi_{\omega}'\psi_{\chi}'\Bigg)
\\
&=\underbrace{\sum_{k=n-1}^{2(n-1)}\lambda^{-\frac{k}{2}}\psi_{k}''}_{S_1} -\underbrace{\sum_{k=n-1}^{2(n-1)}\lambda^{-\frac{k}{2}} \sum_{\substack{\omega+\chi=k \\ -1\leq\omega,\chi\leq n-1}}^{}\psi_{\omega}'\psi_{\chi}'}_{S_{2}}.
\end{aligned}
\end{equation}
By the rule \ref{remarkinduction(ii)}, the fact that $\psi_k=0$ for $k>n-1$ and Lemma~\ref{lem:str}, we have
\begin{equation}
\begin{aligned}
S_1 &\in \lambda^{-\frac{n-1}{2}} \frac{\lambda^{\frac{n-1}{2}}}{(\lambda-V)^{\frac{n-1}{2}}} \sum_{j=0}^{n+1}\frac{T_{j}^{n+1,n+2-j}}{(\lambda-V)^{j}}
= \frac{1}{(\lambda-V)^{\frac{n-1}{2}}} \sum_{j=1}^{n+1} \frac{T_{j}^{n+1,n+2-j}}{(\lambda-V)^{j}}, 
\\
S_2& \in \sum_{k=n-1}^{2(n-1)} \lambda^{-\frac{k}{2}} 
\! \! \! \! \!  \!  \!  \! 
\sum_{\substack{\omega+\chi=k \\ -1\leq\omega,\chi\leq n-1}}
\! \! \! \!  \! \!  
\frac{\lambda^{\frac{\omega}{2}}}{(\lambda-V)^{\frac{\omega}{2}}} \sum_{j_{1}=0}^{\omega+1}\frac{T_{j_{1}}^{\omega+1,\omega+2-j_{1}}}{(\lambda-V)^{j_{1}}}  
\frac{\lambda^{\frac{\chi}{2}}}{(\lambda-V)^{\frac{\chi}{2}}} \sum_{j_{2}=0}^{\chi+1}\frac{T_{j_{2}}^{\chi+1,\chi+2-j_{2}}}{(\lambda-V)^{j_{2}}}    
\\
&=\sum_{k=n-1}^{2(n-1)} \frac{1}{(\lambda-V)^{\frac{k}{2}}} \sum_{\substack{\omega+\chi=k \\ -1\leq\omega,\chi\leq n-1}}
\underbrace{\Bigg( 
\sum_{j_{1}=1}^{\omega+1}\frac{T_{j_{1}}^{\omega+1,\omega+2-j_{1}}}{(\lambda-V)^{j_{1}}}  
\sum_{j_{2}=1}^{\chi+1}\frac{T_{j_{2}}^{\chi+1,\chi+2-j_{2}}}{(\lambda-V)^{j_{2}}}    
\Bigg)}_{S_3}.
\end{aligned}
\end{equation}
Since $\max \{\omega+2-j_{1}\}$, $\max \{\chi+2-j_{2}\} \leq n$ and $ \omega+\chi=k$, by the formula for the product of series, the rule~(d) and Lemma~\ref{lem:T.op} with the maximum in \eqref{eq:T1T2}, we get
\begin{equation}
\begin{aligned}
S_3 & \subset \sum_{j_{1}=1}^{\omega+1}\frac{T_{j_{1}}^{\omega+1,n}}{(\lambda-V)^{j_{1}}}  
\sum_{j_{2}=1}^{\chi+1}\frac{T_{j_{2}}^{\chi+1,n}}{(\lambda-V)^{j_{2}}}
= 
\sum_{j_{1}=0}^{\omega}
\frac{T_{j_{1}+1}^{\omega+1,n}}{(\lambda-V)^{j_{1}+1}}  
\sum_{j_{2}=0}^{\chi}
\frac{T_{j_{2}+1}^{\chi+1,n}}{(\lambda-V)^{j_{2}+1}}
\\&=
\sum_{j=0}^{\omega+\chi} \sum_{q=0}^{j}
\frac{T_{q+1}^{\omega+1,n}}{(\lambda-V)^{q+1}} 
\frac{T_{j-q+1}^{\chi+1,n}}{(\lambda-V)^{j-q+1}}
\subset 
\sum_{j=0}^{\omega+\chi} 
\frac{T_{j+2}^{\omega+\chi+2,n}}{(\lambda-V)^{j+2}}
\\&=
 \sum_{j=2}^{\omega+\chi+2}\dfrac{T_{j}^{\omega+\chi+2,n}}{(\lambda-V)^{j}}
= \sum_{j=2}^{k+2}\dfrac{T_{j}^{k+2,n}}{(\lambda-V)^{j}}.
\end{aligned}
\end{equation}
Observe that $S_3$ does not depend on $\omega$ and $\chi$ wherefore the sum over $\omega$ and $\chi$ in $S_2$ disappears (for it can be included in the constants appearing in~$T$). 

Inserting formulas for $S_1$, $S_2$ and $S_3$ to \eqref{r_nmits1unds2}, we have
\begin{equation}
\begin{aligned}
r_n & \in \frac{1}{(\lambda-V)^{\frac{n-1}{2}}} \sum_{j=1}^{n+1}\frac{T_{j}^{n+1,n+2-j}}{(\lambda-V)^{j}} + \sum_{k=n-1}^{2(n-1)} \frac{1}{(\lambda-V)^{\frac{k}{2}}} \sum_{j=2}^{k+2}\dfrac{T_{j}^{k+2,n}}{(\lambda-V)^{j}}
\\&=
\frac{1}{(\lambda-V)^{\frac{n-1}{2}}}
\underbrace{\sum_{j=1}^{n+1}\frac{T_{j}^{n+1,n+2-j}}{(\lambda-V)^{j}}}_{S_4} + \underbrace{\sum_{k=0}^{n-1}\frac{1}{(\lambda-V)^{\frac{n-1+k}{2}}} \sum_{j=2}^{n+1+k}\dfrac{T_{j}^{n+1+k,n}}{(\lambda-V)^{j}}}_{S_5}.
\end{aligned}
\end{equation}
Writing the first summand of $S_4$ separately, we get 
\begin{equation}
\begin{aligned}
r_n &\in  
\frac{1}{(\lambda-V)^{\frac{n-1}{2}}} \frac{T_{1}^{n+1,n+1}}{\lambda-V} +
\frac{1}{(\lambda-V)^{\frac{n-1}{2}}} 
\sum_{j=2}^{n+1}\frac{T_{j}^{n+1,n+2-j}}{(\lambda-V)^{j}} 
\\
&
\quad + \sum_{k=0}^{n-1} \frac{1}{(\lambda-V)^{\frac{n-1+k}{2}}} 
\sum_{j=2}^{n+1+k} \frac{T_{j}^{n+1+k,n}}{(\lambda-V)^{j}}.
\end{aligned}
\end{equation}
Using the fact that $T_{1}^{n+1,n+1}= \Com V^{(n+1)}$ and writing the first summand of $S_5$ separately, we obtain
\begin{equation}
\begin{aligned}
r_n& \in 
\frac{V^{(n+1)}}{(\lambda-V)^{\frac{n+1}{2}}}+
\frac{1}{(\lambda-V)^{\frac{n-1}{2}}}
\sum_{j=2}^{n+1}\frac{T_{j}^{n+1,n+2-j}}{(\lambda-V)^{j}}
\\& \quad  +\frac{1}{(\lambda-V)^{\frac{n-1}{2}}}
\sum_{j=2}^{n+1}\dfrac{T_{j}^{n+1,n}}{(\lambda-V)^{j}} 
+\sum_{k=1}^{n-1}\dfrac{1}{(\lambda-V)^{\frac{n-1+k}{2}}}
\sum_{j=2}^{n+1+k}\dfrac{T_{j}^{n+1+k,n}}{(\lambda-V)^{j}}.
\end{aligned}
\end{equation}
Observe that $\max\{n+2-j\}\leq n$. Thus, by the rules \ref{remarkinduction(iv)} and \ref{remarkinduction(iii)}, we get
\begin{equation}
\begin{aligned}
r_n& \in \frac{V^{(n+1)}}{(\lambda-V)^{\frac{n+1}{2}}} +
\underbrace{\dfrac{1}{(\lambda-V)^{\frac{n-1}{2}}}
\sum_{j=2}^{n+1}\frac{T_{j}^{n+1,n}}{(\lambda-V)^{j}}}_{\text{the summand for} \ k=0} 
+\sum_{k=1}^{n-1}\dfrac{1}{(\lambda-V)^{\frac{n-1+k}{2}}}
\sum_{j=2}^{n+1+k}\dfrac{T_{j}^{n+1+k,n}}{(\lambda-V)^{j}}
\\
&=\frac{V^{(n+1)}}{(\lambda-V)^{\frac{n+1}{2}}} 
+\sum_{k=0}^{n-1}\dfrac{1}{(\lambda-V)^{\frac{n-1+k}{2}}}
\sum_{j=2}^{n+1+k}\dfrac{T_{j}^{n+1+k,n}}{(\lambda-V)^{j}}.
\end{aligned}
\end{equation}
Hence, the estimate \eqref{rn.est} follows.
\end{proof}

\subsection*{Acknowledgements}
The authors would like to express their gratitude to 
the \emph{American Institute of Mathematics} (AIM) for a support to organise
the workshop~\cite{AIM-2015}, which partially stimulated the present research.
D.K. was partially supported by the GACR grant No.\ 18-08835S.
Until December 2017, the research of P.S.\ was supported by 
the \emph{Swiss National Science Foundation}, 
SNF Ambizione grant No. PZ00P2\_154786.
%
%

%
\bibliography{bib53}
\bibliographystyle{amsplain}

\end{document}